\documentclass[a4paper,10pt]{article}
\usepackage[utf8]{inputenc}

\usepackage{wrapfig}
\usepackage{amsmath} 
\usepackage{amsthm} 
\usepackage{amssymb} 
\usepackage{enumerate} 
\usepackage{esint} 
\usepackage{pgf,tikz} 
\usetikzlibrary{arrows} 
 \usepackage{yfonts} 
 \usepackage{mathrsfs} 
 \usepackage{mathabx} 
 \usepackage{graphicx}
\usepackage{caption}
\usepackage{subcaption}
\usepackage{mathtools} 
\usepackage[titletoc,toc]{appendix} 

\definecolor{ffffff}{rgb}{1.0,1.0,1.0}
\definecolor{qqqqff}{rgb}{0.0,0.0,1.0}
\definecolor{ffqqqq}{rgb}{1.0,0.0,0.0}
\definecolor{zzzzqq}{rgb}{0.6,0.6,0.0}
\definecolor{marronet}{rgb}{0.6,0.2,0}
\definecolor{negre}{rgb}{0,0,0}
\definecolor{vermell}{rgb}{0.8,0.05,0.05}
\definecolor{blau}{rgb}{0.3,0.2,1.}
\definecolor{blauclar}{rgb}{0.,0.,1.}
\definecolor{grisfosc}{rgb}{0.25098039215686274,0.25098039215686274,0.25098039215686274}
\definecolor{verd}{rgb}{0.1,0.6,0.1}
\definecolor{taronja}{rgb}{0.9,0.6,0.05}
\definecolor{vermellclar}{rgb}{1.,0.,0.}
\definecolor{verdet}{rgb}{0,0.8,0.1}
\definecolor{blauverd}{rgb}{0,0.4,0.2}
\definecolor{grisclar}{rgb}{0.6274509803921569,0.6274509803921569,0.6274509803921569}

\newcommand*\circled[1]{\tikz[baseline=(char.base)]{
            \node[shape=circle,draw,inner sep=2pt] (char) {#1};}}

\newcommand*\squared[1]{\tikz[baseline=(char.base)]{
            \node[shape=rectangle,draw,inner sep=2.4pt] (char) {#1}; \node[shape=rectangle,draw,inner sep=1pt] (char) {#1};}}
\newcommand*\squaredGreek[1]{\tikz[baseline=(char.base)]{
            \node[shape=rectangle,draw,inner sep=2.4pt] (char) {$#1$}; \node[shape=rectangle,draw,inner sep=1pt] (char) {$#1$};}}

\newcommand{\R}{{\mathbb R}}       
\newcommand{\N}{{\mathbb N}}       

\newcommand{\diam}{{\rm diam}}
\newcommand{\dist}{{\rm dist}}
\newcommand{\Dist}{{\rm D}}
\newcommand{\Sh}{{\mathbf {Sh}}} 
\newcommand{\SH}{{\mathbf {SH}}} 

\newcommand{\rf}[1]{{(\ref{#1})}}
\newcommand{\norm}[1]{{\left\| {#1} \right\|}}

\newtheorem{theorem}{Theorem}
\newtheorem*{theorem*}{Theorem}
\newtheorem{lemma}[theorem]{Lemma}

\newtheorem{corollary}[theorem]{Corollary}
\newtheorem*{corollary*}{Corollary}

\newtheorem{definition}[theorem]{Definition}

\newtheorem{remark}[theorem]{Remark}

\numberwithin{subsection}{section}
\numberwithin{theorem}{section}
\numberwithin{equation}{section}
\numberwithin{figure}{section}

\usepackage[affil-it]{authblk}

\title{Measuring Triebel-Lizorkin fractional smoothness on domains in terms of first-order differences}

\author{ Mart\'i Prats
\thanks{MP (De\-par\-ta\-ment de Ma\-te\-m\`a\-ti\-ques, U\-ni\-ver\-si\-tat Au\-t\`o\-no\-ma de Bar\-ce\-lo\-na, Catalonia): \texttt{mprats@mat.uab.cat}.}}

\begin{document}
\maketitle
\bibliographystyle{alpha}

\begin{abstract} 
In this note we give equivalent characterizations for a fractional Triebel-Lizorkin space $F^s_{p,q}(\Omega)$ in terms of first-order differences in a uniform domain $\Omega$. The characterization is valid for any positive, non-integer real smoothness $s\in \mathbb{R}_+\setminus \mathbb{N}$ and {indices $1\leq p<\infty$, $1\leq q \leq \infty$}  as long as the fractional part $\{s\}$ is greater than $d/p-d/q$.
\end{abstract}

\section{Introduction}
Let $d\in\N$, $-\infty<s<\infty$, $0<p<\infty$ and $0<q\leq \infty$. A tempered distribution $f$ is said to belong to the Triebel-Lizorkin space $F^s_{p,q}$ whenever the norm
$$\norm{f}_{F^s_{p,q}}=\norm{\norm{\left\{2^{sj}\left(\psi_j \hat{f}\right)\widecheck{\,}\right\}}_{\ell^q}}_{L^p}$$
is finite, where $\hat{\cdot}$ stands for the Fourier transform, $\widecheck{\cdot}$ stands for its inverse, 
and  ${\psi_j}:={\psi_0}(2^{-j}\cdot )-{\psi_0}(2^{-j+1}\cdot )$ {for $j\geq 1$ and }a give{n} radial function ${\psi_0}\in C^\infty_c(B(0,2))$ with ${\psi_0}|_{B(0,1)}\equiv 1$.  This spaces of functions have been studied for several years, a classical reference being Hans Triebel's book \cite{TriebelTheory}. When $s\in \N$ and $q=2$, then $F^{s}_{p,2}$ coincides with the classical Sobolev space $W^{s,p}$ in the sense of equivalent norms, and for any $s>0$ and $q=2$, it coincides with the corresponding Bessel-potential space.

There are many equivalent characterizations for these spaces. We are interested in characterizations in terms of differences in the spirit of the ones introduced by \cite{Strichartz} in the context of Bessel-potential spaces and $0<s<1$, which are suitable for restriction to domains. In \cite[Section 2.5.10]{TriebelTheory} the reader can find characterizations using differences of order $M>s\in \R$ (see \cite[Section 1.11.9]{TriebelTheoryIII} for characterizations dealing with a larger range of admissible indices $s,p,q$ using means in balls). Roughly speaking, one needs to take into account $M+1$ collinear points with constant gap between them. When restricting to a domain, this poses several technical difficulties that can make computations awkward in some contexts. 

However, sometimes it is easier to deal with weak derivatives to avoid using higher order differences. Indeed, combining the lifting property of these spaces, which says that $\norm{f}_{F^s_{p,q}}\approx \norm{f}_{F^{s-1}_{p,q}}+\norm{\nabla f}_{F^{s-1}_{p,q}}$, with some elementary embeddings in \cite[Section 2.3.2]{TriebelTheory} one obtains that whenever $s=k+\sigma$ with $k\in\N$ and $0<\sigma<1$, then
$$\norm{f}_{F^s_{p,q}} \approx \norm{f}_{L^p} + \norm{\nabla^k f}_{F^{\sigma}_{p,q}},$$
where $\nabla^k f$ denotes the vector valued function containing all the weak derivatives of order $k$ as components.
Thus, one can apply only at the last norm the characterization using first order differences, which follows from \cite[Theorem 1.2]{PratsSaksman}.
\begin{corollary}[to {\cite[Theorem 1.2]{PratsSaksman}}]\label{coroPS}
Let $k\geq 0$ and $d\geq 1$ be  natural numbers, let $0<\sigma<1$, let $1\leq p<\infty$, $1\leq q \leq \infty$ with $\sigma >\frac dp-\frac dq$, and call $s:=k+\sigma$. There are constants depending on these parameters such that
$$\norm{f}_{F^s_{p,q}}\approx \norm{f}_{W^{k,p}}+ \left(\int_{\R^d} \left( \int_{\R^d} \frac{|\nabla^kf(x)-\nabla^kf(y)|^q}{|x-y|^{\sigma q+d}}\, dy\right)^\frac pq  dx\right)^\frac1p,$$
for every $f\in W^{k,p}(\R^d)$, with the usual modification whenever $q=\infty$. 
\end{corollary}

In this note we study analogous norms for these spaces in terms of differences on uniform domains.  Let $d\geq 1$ be a natural number, let $0<\sigma<1$, and let {$1\leq p <\infty$, $1\leq q \leq \infty$}. Given a domain $\Omega\subset \R^d$ and $f\in L^1_{loc}(\Omega)$, write
$$\norm{f}_{\dot A^\sigma_{p,q}(\Omega)}:= \left(\int_\Omega \left( \int_\Omega \frac{|f(x)-f(y)|^q}{|x-y|^{\sigma q+d}}\, dy\right)^\frac pq  dx\right)^\frac1p.$$
In a recent paper, Eero Saksman and the author of the present article showed that 
$$\norm{f}_{A^\sigma_{p,q}(\Omega)}:=\norm{f}_{L^p(\Omega)} + \norm{ f}_{\dot A^\sigma_{p,q}(\Omega)}.$$
is an equivalent norm for the space $F^{\sigma}_{p,q}(\Omega)$ whenever $\Omega$ is a uniform domain, $\sigma >\frac dp-\frac dq$  {(in Theorem \ref{theoEndpoint} in the appendix we discuss the endpoint cases and the unbounded domains)}. This characterization allowed them to show a $T(1)$-type theorem for the boundedness of convolution Calder\'on-Zygmund operators in $F^\sigma_{p,q}(\Omega)$ {when $1<p,q<\infty$}.  

In this note we show that there are equivalent characterizations for all the positive non-integer orders of smoothness:
\begin{theorem}\label{theoEquivalent}
Let $k\geq 0$ and $d\geq 1$ be  natural numbers, let $0<\sigma<1$, let {$1\leq p <\infty$, let $1\leq q \leq \infty$} with $\sigma >\frac dp-\frac dq$, and call $s:=k+\sigma$. Given a uniform domain $\Omega\subset \R^d$ and $f\in L^1_{loc}(\Omega)$, the norms
$$\norm{f}_{A^s_{p,q}(\Omega)}:=\norm{f}_{W^{k,p}(\Omega)} + \sum_{|\alpha|=k}\norm{D^\alpha f}_{\dot A^\sigma_{p,q}(\Omega)} ,$$
where $\alpha$ takes values in $\N^d$ with $|\alpha|:=\sum\alpha_j=k$, and
$$\norm{f}_{F^s_{p,q}(\Omega)}:= \inf \left\{ \norm{g}_{F^s_{p,q}(\R^d)} : g\in F^s_{p,q}(\R^d)  \mbox{ with } g|_\Omega \equiv f\right\}$$
are equivalent for the Triebel-Lizorkin space $F^s_{p,q}(\Omega)$, with constants depending on $s,p,q,d$ and the uniformity constants of $\Omega$.
\end{theorem}
This norm has a self-improvement property from \cite[Theorems 1.5 and 1.6]{PratsSaksman} {(in Theorems \ref{theoShadow} and \ref{theoBall} we discuss the endpoint cases and the unbounded domains)}. Call $\delta(x)=\dist(x,\partial\Omega)$. Consider the Carleson boxes (or shadows) $\mathbf{Sh}(x):=\{y\in\Omega : |y-x|\leq c_{\Omega}\delta(x)\}$ for a certain constant $c_\Omega>1$ big enough. Then we have the following reduction for the Triebel-Lizorkin norm:
\begin{corollary}\label{coroNormOmega}
Let $\Omega\subset \R^d$ be a {uniform} domain, let $0<\sigma< 1$, $k\in\N$, {$1\leq p<\infty$ and $1\leq q \leq \infty$}  with $\sigma>\frac{d}{p}-\frac{d}{q}$. 
Then $f\in F^s_{p,q}(\Omega)$ for $s=k+\sigma$ if and only if 
\begin{equation*}
\norm{f}_{W^{k,p}(\Omega)} + \sum_{|\alpha|=k} \left(\int_\Omega \left(\int_{\Sh(x)}\frac{|D^\alpha f(x)-D^\alpha f(y)|^q}{|x-y|^{\sigma q+d}} \,dy\right)^{\frac{p}{q}}dx\right)^{\frac{1}{p}}<\infty.
\end{equation*}
and the norms are equivalent.
\end{corollary}

The self-improvement is stronger when $p\geq q$, when we can restrict to Whitney balls:
\begin{corollary}\label{coroNormOmegapgtrq}
Let $\Omega\subset \R^d$ be a {uniform} domain, let  $0<\sigma< 1$, $k\in\N$, {$1\leq q\leq p <\infty$} and $0<\rho<1$. 
Then $f\in F^s_{p,q}(\Omega)$ for $s=k+\sigma$ if and only if 
\begin{equation*}
\norm{f}_{W^{k,p}(\Omega)} + \sum_{|\alpha|=k} \left(\int_\Omega \left(\int_{B\left(x,\rho \delta(x)\right)}\frac{|D^\alpha f(x)-D^\alpha f(y)|^q}{|x-y|^{\sigma q+d}} \,dy\right)^{\frac{p}{q}}dx\right)^{\frac{1}{p}}<\infty.
\end{equation*}
and the norms are equivalent.
\end{corollary}

The last result has {\cite[Proposition 5]{Dyda}} as a particular case ($\Omega$ a Lipschitz domain, $p=q$, $k=0$). 

In a forthcoming paper these norms will be used to study the relation between the Triebel-Lizorkin regularity of quasiconformal mappings between domains, the regularity of their Beltrami coefficient and the regularity of the boundary of the domains{.}

The crucial estimate to show Theorem \ref{theoEquivalent} is to find an extension operator for the norm $A^s_{p,q}$.  That is, we need to find a  bounded linear operator $\Lambda: A^s_{p,q}(\Omega)\to A^s_{p,q}(\R^d)$ such that $(\Lambda f)|_\Omega=f$. Once this is settled, the theorem follows by classical estimates. Here we will recover the extension operators defined by Peter Jones in \cite{Jones} (the reader will note that we write $\Lambda_k$ where Peter Jones wrote $\Lambda_{k+1}$).
\begin{theorem}\label{theoExtension}
Let $\Omega$ be a uniform domain and $k\in \N$. There exists a linear operator $\Lambda_k: W^{k+1,\infty}(\Omega) \to W^{k+1,\infty}(\R^d)$ such that for every 
$0<\sigma< 1$, {$1\leq p<\infty$ and $1\leq q \leq \infty$}  with $\sigma>\frac{d}{p}-\frac{d}{q}$, then 
$$\Lambda_k:A^s_{p,q}(\Omega)\to A^s_{p,q} (\R^d)$$
(with $s=\sigma+k$) is a bounded {extension} operator.
\end{theorem}
Of course, the reader may find in the literature other extension operators acting on $F^s_{p,q}(\Omega)$, mainly when $\Omega$ is a Lipschitz domain. For instance \cite[Theorem 2.2]{Rychkov} presents an operator which acts continuously in a family of Triebel-Lizorkin and Besov spaces with unbounded regularity parameter. From Theorem \ref{theoEquivalent} it follows that these extension operators are also continuous in $A^{s}_{p,q}(\Omega)$ for Lipschitz domains. 

{Similar problems can be considered for Besov domains. In \cite{Dispa} the reader can find a characterization in terms of differences for Besov spaces on Lipschitz domains. It would be interesting to know whether this results have or do not have counterparts for Besov spaces on uniform domains. In \cite{Seeger} Seeger could find characterizations for the Triebel-Lizorkin norms on uniform domains using means on balls of higher order differences of the function. It would be interesting to find equivalent characterizations using first order differences of derivatives of the function, and such that the size of the ball varies with the distance of its center to the boundary, question that can also be studied in the Besov space.}

Section \ref{secUniform} is devoted to define uniform domains and to recall the main properties of their Whitney coverings, regarding sums on chains of cubes (denoted Cigars in \cite{Vaisala}) and shadows (commonly known as ``Carleson boxes''). 

Section \ref{secExtension} is the core of the present note. First the Jones' extension operator via Meyers' polynomials is introduced. This is followed by a lemma that settles a key estimate where the differences between $p<q$, $p=q$ and $p>q$ are overcome. After that the reader will find the proof of Theorem \ref{theoExtension}, divided in two parts. First the $W^{k,p}(\Omega)$ character is established and the weak derivative of the extension operator is given using some estimates from \cite{Jones}. Finally, using \cite{PratsSaksman} the boundedness in $A^s_{p,q}(\Omega)$ is reduced to controlling a series of ``error terms'' which are settled  using all the machinery developed in the aforementioned papers. 

Finally Section \ref{secMain} contains the proofs of Corollary \ref{coroPS} and Theorem \ref{theoEquivalent}.

{In the appendix we show how to extend the results in \cite{PratsSaksman} to the setting of unbounded domains and with the indices $p, q$ reaching the endpoints.}

Throughout the paper we will not pay much attention to the particular value of constants. Thus, $A\lesssim B$ means that there exists a universal constant $C$ such that $A\leq CB$. If we want to stress the dependence of the constant in certain parameters, for instance $\sigma$ and $p$, we will write $A\lesssim_{\sigma,p} B$ or $A\leq C_{\sigma, p} B$.

\section{Uniform domains}\label{secUniform}
\begin{figure}[h]
 {\centering
  \includegraphics[width=0.65 \linewidth]{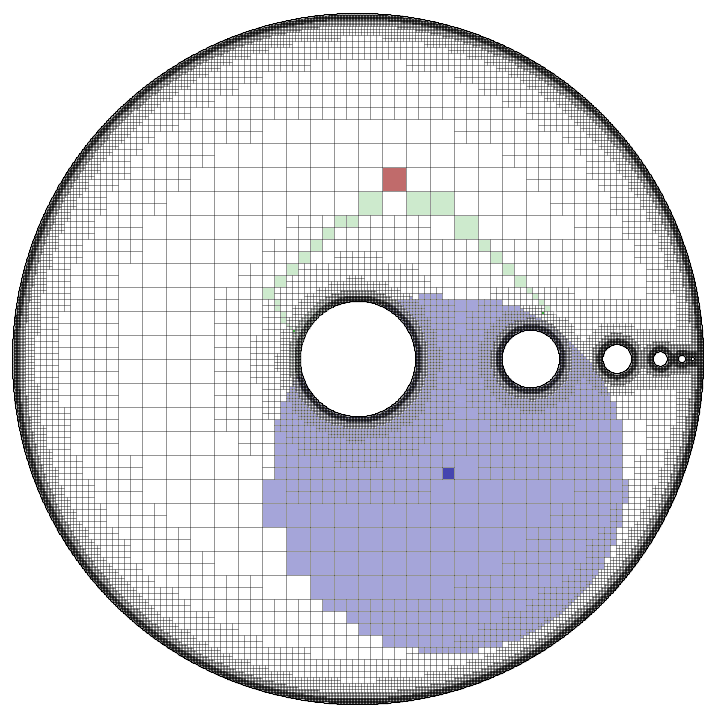}  
\caption{A uniform domain with a Whitney covering. In the upper part there is an admissible chain joining two cubes (the central one shaded), on the lower part the shadow of another cube.}\label{figWhitney}}
\end{figure}

\begin{definition}\label{defWhitney}
Given a domain $\Omega$, we say that a collection of open dyadic cubes $\mathcal{W}$ is a {\rm Whitney covering} of $\Omega$ if they are disjoint, the union of the cubes and their boundaries is $\Omega$, there exists a constant $C_{\mathcal{W}}$ such that 
$$C_\mathcal{W} \ell(Q)\leq \Dist(Q, \partial\Omega)\leq 4C_\mathcal{W}\ell(Q),$$
and the family $\{50 Q\}_{Q\in\mathcal{W}}$ has finite superposition. Moreover, we will assume that 
\begin{equation}\label{eqWhitney5}
S\subset 5Q \implies \ell(S)\geq \frac12 \ell(Q).
\end{equation}
\end{definition}
The existence of such a covering is granted for any open set different from $\R^d$ and in particular for any domain as long as $C_\mathcal{W}$ is big enough (see \cite[Chapter 1]{SteinPetit} for instance).

\begin{definition}\label{defEpsilonAdmissible}
Let $\Omega$ be a domain, $\mathcal{W}$ a Whitney decomposition of $\Omega$ and $Q,S\in\mathcal{W}$.  Given $M$ cubes $Q_1,\dots,Q_M\in\mathcal{W}$ with $Q_1=Q$ and $Q_M=S$, the $M$-tuple $(Q_1,\dots,Q_M)\in\mathcal{W}^M$  is a {\em chain} connecting $Q$ and $S$ if the cubes $Q_j$ and $Q_{j+1}$ are neighbors for $j<M$. We write $[Q,S]=(Q_1,\dots,Q_M)$ for short.

Let $\varepsilon\in\R$. We say that the chain $[Q,S]$ is {\em $\varepsilon$-admissible} if 
\begin{itemize}
\item the \emph{length}  of the chain is bounded by
\begin{equation}\label{eqLengthDistance}
\ell([Q,S]):=\sum_{j=1}^M\ell(Q_j)\leq \frac1\varepsilon\Dist(Q,S)
\end{equation}
\item and there exists $j_0<M$ such that the cubes in the chain satisfy
\begin{equation}\label{eqAdmissible1}
\ell(Q_j)\geq\varepsilon \Dist(Q_1,Q_j) \mbox{ for all } j\leq j_0 \mbox{\quad\quad  and \quad\quad }
\ell(Q_j)\geq\varepsilon \Dist(Q_j,Q_M) \mbox{ for all } j\geq j_0 .
\end{equation}
\end{itemize}
The $j_0$-th cube, which we call \emph{central}, satisfies that $\ell(Q_{j_0})\gtrsim_d \varepsilon \Dist(Q,S)$ by \rf{eqAdmissible1} and the triangle inequality. We will write  $Q_S=Q_{j_0}$. Note that this is an abuse of notation because the central cube of $[Q,S]$ may vary for different $\varepsilon$-admissible chains joining $Q$ and $S$.

We write (abusing notation again) $[Q,S]$ also for the set $\{Q_j\}_{j=1}^M$. Thus, we will write $P\in[Q,S]$ if $P$ appears in a coordinate of the $M$-tuple $[Q,S]$.
\end{definition}

Consider a domain $\Omega$ with covering $\mathcal{W}$ and two cubes $Q,S\in\mathcal{W}$ with an $\varepsilon$-admissible chain $[Q,S]$. From Definition \ref{defEpsilonAdmissible} it follows that
\begin{equation}\label{eqAdmissible2}
\Dist(Q,S)\approx_{\varepsilon,d} \ell([Q,S])\approx_{\varepsilon,d} \ell(Q_S).
\end{equation}
%
 
\begin{definition}\label{defUniform}
We say that a domain $\Omega\subset\R^d$ is a {\em uniform domain} if there exists a Whitney covering $\mathcal{W}$ of $\Omega$ and $\varepsilon \in\R$ such that for any pair of cubes $Q,S \in\mathcal{W}$, there exists an $\varepsilon$-admissible chain $[Q,S]$. Sometimes  will write {\em $\varepsilon$-uniform domain} to fix the constant $\varepsilon$ (see Figure \ref{figWhitney}).
\end{definition}
 
For $1\leq j_1\leq j_2\leq M$, the subchain $[Q_{j_1},Q_{j_2}]_{[Q,S]} \subset[Q,S]$ is defined as $(Q_{j_1},Q_{j_1+1},\dots,Q_{j_2})$. We will write $[Q_{j_1},Q_{j_2}]$ if there is no risk of confusion.
Now we can define the shadows:
\begin{definition}\label{defShadow}
Let $\Omega$ be an $\varepsilon$-uniform domain with Whitney covering $\mathcal{W}$. 
Given a cube $P\in\mathcal{W}$ centered at $x_P$ and a real number $\rho$,  the {\em $\rho$-shadow} of $P$ is the collection of cubes
$$\SH_\rho(P)=\{Q\in\mathcal{W}:Q\subset B(x_P,\rho\,\ell(P))\}, $$
and its  {\em ``realization''} is the set
$$\Sh_{\rho}(P)=\bigcup_{Q\in\SH_\rho(P)} Q.$$

By the previous remark and the properties of the Whitney covering, we can define $\rho_\varepsilon>1$ such that the following properties hold:
\begin{itemize}
\item For every $\varepsilon$-admissible chain $[Q,S]$, and every $P\in[Q,Q_S]$ we have that $Q\in\SH_{\rho_\varepsilon}(P)$.
\item Moreover, every cube $P$ belonging to an $\varepsilon$-admissible chain $[Q,S]$ belongs to the shadow $\SH_{\rho_\varepsilon}(Q_S)$.
\end{itemize}
\end{definition}

\begin{remark}[see {\cite[Remark 2.6]{PratsSaksman}}]
\label{remInTheShadow}
Given an $\varepsilon$-uniform domain $\Omega$ we will write $\Sh$ for $\Sh_{\rho_\varepsilon}$. We will write also $\SH$ for $\SH_{\rho_{\varepsilon}}$.

For $Q\in\mathcal{W}$ and $s>0$,  we have that  
\begin{equation}\label{eqAscendingToGlory}
 \sum_{L: Q\in \SH(L)}\ell(L)^{-s} \lesssim \ell(Q)^{-s} 
 \end{equation}
and, moreover, if $Q\in\SH(P)$, then
\begin{equation}\label{eqAscendingPath}
 \sum_{L\in[Q,P]}\ell(L)^{s} \lesssim \ell(P)^s \mbox{\quad\quad and \quad\quad}  \sum_{L\in[Q,P]}\ell(L)^{-s}\lesssim \ell(Q)^{-s} .
 \end{equation}
\end{remark}

We recall the definition of the non-centered Hardy-Littlewood maximal operator. Given $f\in L^1_{loc}(\R^d)$ and $x\in\R^d$, we define $Mf(x)$ as the supremum of the mean of $f$ in cubes containing $x$, that is,
$$Mf(x)=\sup_{Q:  x\in Q} \frac{1}{|Q|} \int_Q f(y) \, dy.$$
It is a well known fact that this operator is bounded in $L^p$ for $1<p<\infty$.
The following lemma is proven in \cite{PratsTolsa} and will be used repeatedly along the proofs contained in the present text.

\begin{lemma}\label{lemMaximal}
Let $\Omega$ be a uniform domain with an admissible Whitney covering $\mathcal{W}$. Assume that $g\in L^1(\Omega)$ and $r>0$. For every $\eta>0$, $Q\in\mathcal{W}$ and $x\in \R^d$, we have
\begin{enumerate}[1)]
\item The non-local {inequalities} for the maximal operator
	\begin{equation}\label{eqMaximalFar}
	 \int_{|y-x|>r} \frac{g(y) \, dy}{|y-x|^{d+\eta}}\lesssim_d \frac{Mg(x)}{r ^\eta}
\mbox{\quad\quad and \quad\quad}
	 \sum_{S:\Dist(Q,S)>r}  \frac{\int_S g(y) \, dy}{\Dist(Q,S)^{d+\eta}}\lesssim_d \frac{\inf_{y\in Q} Mg(y)}{r ^\eta}.
	 \end{equation}
\item The local {inequalities} for the maximal operator
	\begin{equation}\label{eqMaximalClose}
	 \int_{|y-x|<r} \frac{g(y) \, dy}{|y-x|^{d-\eta}}\lesssim_d r ^\eta Mg(x)
\mbox{\quad\quad and \quad\quad}
	\sum_{S:\Dist(Q,S)<r}  \frac{\int_S g(y) \, dy}{\Dist(Q,S)^{d-\eta}}\lesssim_d \inf_{y\in Q} Mg(y) \,r^\eta.
	 \end{equation}
\item In particular we have
	\begin{equation}\label{eqMaximalAllOver}
		\sum_{S\in\mathcal{W}} \frac{\ell(S)^d}{\Dist(Q,S)^{d+\eta}} \lesssim_d \frac{1}{\ell(Q)^\eta}
\mbox{\quad\quad and \quad\quad}
		\sum_{S\in \SH_{{\rho}}(Q)} \ell(S)^{d} \lesssim_{d,\rho} \ell(Q)^d
	\end{equation}
and, by Definition \ref{defShadow},
	\begin{equation}\label{eqMaximalGuay}
	\sum_{S\in\SH_{{\rho}} (Q)} \int_S g(x) \, dx\lesssim_{d,\rho} \inf_{y\in Q} Mg(y) \, \ell(Q)^d.
	 \end{equation}
\end{enumerate}
\end{lemma}

\section{Peter Jones' extension}\label{secExtension}
\begin{figure}[h]
 {\centering
  \includegraphics[width=0.5 \linewidth]{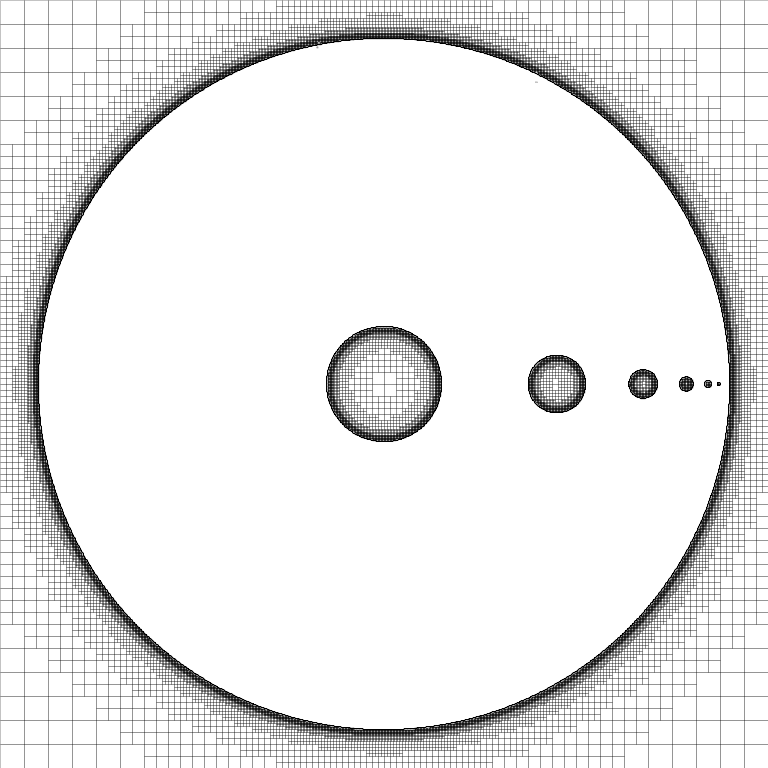}  
\caption{The covering $\mathcal{W}_2$ of $\overline{\Omega}^c$.}\label{figW2}}
\end{figure}

Consider a given $\varepsilon$-uniform domain $\Omega$. In \cite{Jones} Peter Jones defines an extension operator $\Lambda_k:W^{k+1,p}(\Omega) \to W^{k+1,p}(\R^d)$ for  $1\leq p \leq \infty$. This extension operator is used to prove that the intrinsic characterization of $W^{k+1,p}(\Omega)$ given by
$$\norm{f}_{W^{k+1,p}(\Omega)}\approx \norm{f}_{L^p(\Omega)}+\norm{\nabla^{k+1} f}_{L^p(\Omega)}$$
is equivalent to the restriction norm. Next we will see that the same operator is an extension operator for $A^{k+\sigma}_{p,q}(\Omega)$ for $0<\sigma<1$ with $\sigma>\frac{d}{p}-\frac{d}{q}$. 

To define it we need a Whitney covering $\mathcal{W}_1$ of $\Omega$, a Whitney covering $\mathcal{W}_2$ of $\bar{\Omega}^c$ (see Figure \ref{figW2}), and we define $\mathcal{W}_3$ to be the collection of cubes in $\mathcal{W}_2$ with side-lengths small enough, say $\ell(Q)\leq \ell_0$, so that for any $Q\in \mathcal{W}_3$ there is a cube $S\in \mathcal{W}_1$ with $\Dist(Q,S)\leq C \ell(Q)$ and $\ell(Q)=\ell(S)$ (see \cite[Lemma 2.4]{Jones}). We define the symmetrized cube $Q^*$ as one of the cubes satisfying these properties (see Figure \ref{figSymmetrization}). Note that the number of possible choices for $Q^*$ is uniformly bounded and, if $\Omega$ is an unbounded uniform domain, then {$\ell_0<\infty$ can be chosen freely}.
\begin{figure}[h]
 {\centering
  \includegraphics[width=0.6 \linewidth]{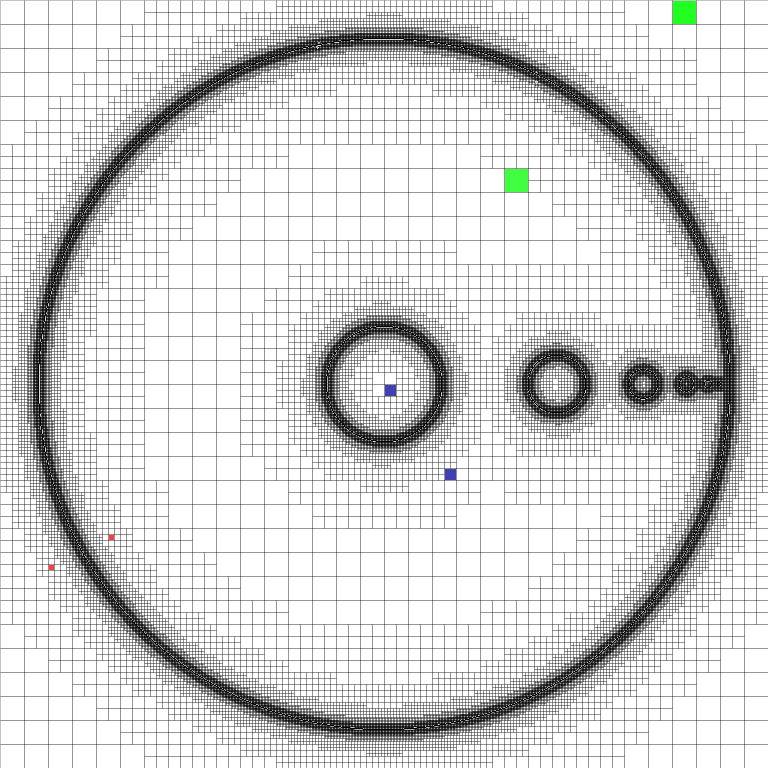}  
\caption{Pairs of cubes $Q$ and $Q^*$ by the symmetrization relation.}\label{figSymmetrization}}
\end{figure}

\begin{lemma}\label{lemSymmetrized}[see \cite{Jones}]
For cubes $Q_1,Q_2\in\mathcal{W}_3$ and $S\in\mathcal{W}_1$ we have that
\begin{itemize}
\item The symmetrized cubes have finite overlapping: there exists a constant $C$ depending on the parameter $\varepsilon$ and the dimension $d$ such that $\#\{Q\in\mathcal{W}_3: Q^*=S\}\leq C$.
\item The long distance is invariant in the following sense:
\begin{equation}\label{eqLongDistanceInvariant}
\Dist(Q_1^*,Q_2^*)\approx \Dist(Q_1,Q_2) \mbox{\quad\quad and \quad\quad}\Dist(Q_1^*,S)\approx \Dist(Q_1,S) 
\end{equation}
\item In particular, if ${Q_1}\cap2{Q_2}\neq \emptyset$ ($Q_1$ and $Q_2$ are neighbors by \rf{eqWhitney5}), then $\Dist(Q_1^*,Q_2^*)\approx \ell(Q_1)$.
\end{itemize}
\end{lemma}

We define the family of bump functions $\{\psi_Q\}_{Q\in \mathcal{W}_2}$ to be a partition of the unity associated to $\left\{\frac{11}{10}Q\right\}_{Q\in\mathcal{W}_2}$, that is, their sum $\sum\psi_Q\equiv 1$, they satisfy the pointwise inequalities $0\leq \psi_Q\leq \chi_{\frac{11}{10}Q}$ and $\norm{\nabla^j\psi_Q}_\infty \lesssim \frac{1}{\ell(Q)^j}$ for $j\leq k$.

Norman G. Meyers introduced a collection of projections $L:W^{k,p}(Q) \to \mathcal{P}^k$ in \cite{Meyers} which allows us to iterate the Poincar\'e inequality. Peter Jones uses the following particular simple case, whose existence is granted by elementary linear algebra:
\begin{definition}
Let  $Q \subset \R^d$. Given $f\in L^1(Q)$ with weak derivatives up to order $k$, we define $\mathbf{P}^{k}_Q f\in \mathcal{P}^{k}$ as the unique polynomial of degree smaller or equal than $k$ such that 
\begin{equation}\label{eqdefpnQ}
\int_{Q} D^\beta \mathbf{P}_Q^k f \,dm=\int_{Q} D^\beta f\, dm
\end{equation}
for every multiindex $\beta \in \N^d$ with $|\beta| \leq k$.
\end{definition}

\begin{lemma}[see {\cite[Lemma 4.2]{PratsTolsa}}]\label{lempoly}
Given a cube $Q$ and $f\in W^{k,1}(Q)$, the polynomial $\mathbf{P}^{k}_{Q} f\in \mathcal{P}^{k}$ exists and is unique.
Furthermore, this polynomial has the following properties:
\begin{enumerate}
\item Let $x_Q$ be the center of $Q$. If we consider the Taylor expansion of $\mathbf{P}_{Q}^{k} f$ at $x_Q$, 
\begin{equation}\label{eqtaylorexp}
\mathbf{P}_{Q}^{k} f (y)= \sum_{\substack{\gamma\in\N^d \\ |\gamma|\leq k}}m_{Q,\gamma} (y-x_Q)^\gamma,
\end{equation}
then the coefficients $m_{Q,\gamma}$ are bounded by
\begin{equation}\label{eqP1}
|m_{Q,\gamma}| \leq c_k  \sum_{j=|\gamma|}^{k} \norm{\nabla^{j} f}_{L^1(Q)} \ell(Q)^{j-|\gamma|-d}. 
\end{equation}
In particular, 
\begin{equation}\label{eqP1Bis}
\norm{P_{Q}^{k} f}_{L^p(Q)}\leq c_k  \sum_{j=0}^{k}  \ell(Q)^j \norm{\nabla^{j} f}_{L^p(Q)} \quad\quad \mbox{for } 1\leq p \leq \infty. 
\end{equation}

\item Furthermore, if $f\in W^{k,p}(Q)$, for ${1\leq p \leq \infty}$ we have
 \begin{equation}\label{eqP2}
\|f-\mathbf{P}_{Q}^{k} f\|_{L^p(Q)}\leq C \ell(Q)^k \norm{\nabla^k f - (\nabla^k f)_{Q}}_{L^p(Q)} .
 \end{equation}
{Here and through all the text $(f)_Q$ will denote the mean of $f$ in a cube $Q$, with $f$ possibly vector-valued.}

\item Given a uniform domain $\Omega$ with Whitney covering $\mathcal{W}$, given $\beta \in \N_0^d$ with $|\beta|\leq k$ and given two Whitney cubes $Q, S\in \mathcal{W}$ and $f\in W^{k,p}(\Omega)$, 
\begin{equation}\label{eqP3}
\norm{D^\beta (\mathbf{P}^{k}_{S} f-\mathbf{P}^{k}_{Q} f)}_{L^p(S)}  \leq \sum_{P\in [S,Q]}\frac{\ell(S)^\frac dp \Dist(P,S)^{k-|\beta|}}{\ell(P)^\frac dp}\norm{\nabla^k f - (\nabla^k f)_P}_{L^p(P)}.
\end{equation}
%
%
\end{enumerate}
\end{lemma}

  We can define the operator
$$\Lambda_k f(x)= f(x)\chi_\Omega(x) + \sum_{Q\in\mathcal{W}_3} \psi_Q(x) P^k_{Q^*} f(x) \mbox{ for any }f\in W^{1,k}_{loc}(\Omega).$$
{This} function is defined almost everywhere because the boundary of the domain $\Omega$ has zero Lebesgue measure (see \cite[Lemma 2.3]{Jones}).

\begin{proof}[{Proof of Theorem \ref{theoExtension}}]
Let $f\in A^s_{p,q}(\Omega)$. We want to check that 
${\Lambda_k}f\in W^{k,p}(\R^n)$ and 
$$\norm{\Lambda_k f}_{A^s_{p,q}(\R^n)}\leq C\norm{f}_{A^s_{p,q}(\Omega)}.$$

The case $k=0$ is shown in \cite[Theorem 1.4]{PratsSaksman}. Although it is not in the statement of that theorem, its  proof {can be extended to} the case of unbounded $\Omega$ {and $q=\infty$, see Theorem \ref{theoEndpoint} in the Appendix.}

Let us assume that $k\geq 1$, and consider $\alpha \in \N_0^d$ with $|\alpha| \leq k$. First we check that the distributional derivative $D^\alpha \Lambda_k f \in L^p(\R^d)$ and it coincides with $\chi_{\Omega} D^\alpha f + \sum_{Q\in \mathcal{W}_3} D^\alpha (\psi_Q {P}^k_{Q^*})$. To do so, for every cube $Q$ consider the polynomial 
$$\mathring{P}^k_Q:=P^k_Q-P^{k-1}_Q,$$
and let 
$$\mathring{\Lambda}^k f := \Lambda^k f - \Lambda^{k-1}f = \sum_{Q\in \mathcal{W}_3} \psi_Q \mathring{P}^k_{Q^*}.$$
Since $f\in W^{k,p}(\Omega)$ and $\Lambda^{k-1}: W^{k,p}(\Omega)\to W^{k,p}(\R^d)$, we have that $\Lambda^{k-1}f \in W^{k,p}(\R^d)$
and since the boundary of $\Omega$ has Lebesgue measure zero, the weak derivative coincides with its restrictions to $\Omega$ and $\overline{\Omega}^c$, that is, 
$$D^\alpha \Lambda^{k-1}f = \chi_\Omega D^\alpha_\Omega f + \chi_{\overline{\Omega}^c} D^\alpha_{\overline{\Omega}^c} \Lambda^{k-1}f  = \chi_\Omega D^\alpha f + \sum_{Q\in \mathcal{W}_3} D^\alpha (\psi_Q {P}^{k-1}_{Q^*}),$$
where we denoted $D^\alpha_U$ for the weak derivative on an open set $U$.
Thus, we only need to check that 
\begin{equation}\label{eqWeakDerivativeIntegrable}
\sum_{Q\in \mathcal{W}_3} D^\alpha (\psi_Q \mathring{P}^k_{Q^*})  \mbox{\quad \quad is an $L^p$ function}
\end{equation}
 for every such $\alpha$, and then show that
\begin{equation}\label{eqWeakDerivativeCoincides}
D^\alpha \mathring{\Lambda}^{k}f = \sum_{Q\in \mathcal{W}_3} D^\alpha (\psi_Q \mathring{P}^k_{Q^*}).
\end{equation}

 For multiindices $\alpha$ and $\beta$ we say $\beta<\alpha$ whenever $\beta_j< \alpha_j$ for $1\leq j\leq d$. Now, given $\beta<\alpha$ and $Q\in \mathcal{W}_1$, we have that 
$$\int_Q D^\beta P^k_Q f =\int_Q D^\beta f =\int_Q D^\beta P^{k-1}_Q f,$$
and applying Poincar\'e inequality recursively and \rf{eqP1Bis}, we obtain
$$\norm{D^\beta \mathring{P}^k_Q f}_{L^p(Q)}\leq \ell(Q)^{|\alpha-\beta|}\norm{\nabla^{|\alpha|} \mathring{P}^k_Q f}_{L^p(Q)}\lesssim \ell(Q)^{|\alpha-\beta|} {(1+\ell(Q))^{k-|\alpha|} } \norm{\nabla^{|\alpha|} f}_{W^{k-|\alpha|,p}(Q)}.$$
Thus, using the finite overlapping of the enlarged cubes of the Whitney covering, the equivalence of the norms of polynomials and the previous fact, we get
\begin{align*}
\int_{\R^d} 
	& \left| \sum_{Q\in \mathcal{W}_3} D^\alpha (\psi_Q \mathring{P}^k_{Q^*}f) (x) \right|^p dx
		\lesssim \sum_{Q\in \mathcal{W}_3} \sum_{\beta\leq \alpha} \int_{\frac{11}{10}Q} \ell(Q)^{-|\alpha-\beta|p} |D^\beta\mathring{P}^k_{Q^*}f (x)|^p dx \\
	& \lesssim \sum_{Q\in \mathcal{W}_3} \sum_{\beta\leq \alpha}  \ell(Q)^{-|\alpha-\beta|p} \norm{D^\beta\mathring{P}^k_{Q^*}f}_{L^p(Q^*)}^p
		 \lesssim \sum_{Q\in {\mathcal{W}_1}} \sum_{\beta\leq \alpha} {(1+\ell_0)^{(k-|\alpha|)p} }  \norm{\nabla^{|\alpha|} f}_{W^{k-|\alpha|,p}(Q)}^p\\
	& \lesssim \norm{\nabla^{|\alpha|} f}_{W^{k-|\alpha|,p}(\Omega)}^p,
\end{align*}
showing \rf{eqWeakDerivativeIntegrable}.

Now, this boundedness also implies that, given $\varphi\in C^\infty_c(\R^d)$ and $\beta\leq \alpha$, we get
$$\int_{\R^d} D^{\beta}\varphi (x)\sum_Q D^{\alpha-\beta} (\psi_Q \mathring P^k_{Q^*}f) (x)\,dx= \sum_Q \int_{\frac{11}{10}Q} D^{\beta}\varphi (x) D^{\alpha-\beta} (\psi_Q\mathring P^k_{Q^*}f)(x)\, dx$$
because the integral is absolutely convergent. Thus, 
\begin{align*}
(-1)^{|\alpha|} \langle \mathring \Lambda^k f,  D^\alpha \varphi \rangle
	& = (-1)^{|\alpha|} \int_{\R^d} D^{\alpha} \varphi \sum_Q (\psi_Q \mathring P^k_{Q^*}f)
		= (-1)^{|\alpha|} \sum_Q \int_{\frac{11}{10}Q} D^{\alpha}\varphi \, \psi_Q \mathring P^k_{Q^*}f\\
	& = \sum_Q \int_{\frac{11}{10}Q} \varphi D^{\alpha}(\psi_Q \mathring P^k_{Q^*}f) 
		=  \int_{\Omega} \varphi \sum_Q D^{\alpha}(\psi_Q \mathring P^k_{Q^*}f) 
\end{align*}
where all the integrals are taken with respect to the Lebesgue measure, and \rf{eqWeakDerivativeCoincides} follows.

 It remains to show that
$$\norm{D^\alpha \Lambda_k f}_{\dot{A}^\sigma_{p,q}(\R^n)}\leq C\norm{f}_{A^s_{p,q}(\Omega)}$$
for $|\alpha|=k$. Since 
\begin{align}\label{eqLambda0}
 D^\alpha \Lambda_k f 
\nonumber	& =  D^\alpha f  \chi_\Omega + \sum_{Q\in\mathcal{W}_3} D^\alpha (\psi_Q P^k_{Q^*} f)
		=   D^\alpha f  \chi_\Omega + \sum_{Q\in\mathcal{W}_3} \sum_{\beta \leq \alpha}{\alpha \choose \beta} D^{\alpha-\beta}\psi_Q D^\beta  P^k_{Q^*} f\\
	& 	=   \Lambda_0 (D^\alpha f) +\sum_{\beta < \alpha}{\alpha \choose \beta} \sum_{Q\in\mathcal{W}_3}  D^{\alpha-\beta}\psi_Q   D^\beta P^k_{Q^*} f.
\end{align}

Now, from \cite[Theorem 1.4]{PratsSaksman} we already have 
$$\norm{\Lambda_0 (D^\alpha f)}_{A^\sigma_{p,q}(\R^n)}\leq C \norm{D^\alpha f}_{A^\sigma_{p,q}(\Omega)}\leq C \norm{ f}_{A^s_{p,q}(\Omega)}$$
{(see the appendix for the case of unbounded $\Omega$ or $q=\infty$)}. 
Thus, for every $|\beta|<k$ we  need to control
\begin{align}\label{eqBreakBeta}
\squaredGreek{\beta}
\nonumber	& :=\norm{\sum_{P\in\mathcal{W}_3} D^{\alpha-\beta}\psi_P D^\beta  P^k_{P^*} f}_{\dot A^\sigma_{p,q}(\R^n)}^p\\
\nonumber	& =  \int_\Omega\left(\int_{\Omega^c} \frac{|\sum_{P\in\mathcal{W}_3} D^{\alpha-\beta}\psi_P(y) D^\beta  P^k_{P^*} f(y) |^q}{|x-y|^{\sigma q +d}} \, dy\right)^\frac pq dx\\
\nonumber	& \quad +  \int_{\Omega^c}\left(\int_{\Omega} \frac{|\sum_{P\in\mathcal{W}_3} D^{\alpha-\beta}\psi_P(x) D^\beta  P^k_{P^*} f(x) |^q}{|x-y|^{\sigma q +d}} \, dy\right)^\frac pq dx\\
\nonumber	& \quad + \int_{\Omega^c} \left(\int_{\Omega^c} \frac{|\sum_{P\in\mathcal{W}_3} \left((D^{\alpha-\beta}\psi_P D^\beta  P^k_{P^*} f)(x) - (D^{\alpha-\beta}\psi_P D^\beta  P^k_{P^*} f)(y)\right)|^q}{|x-y|^{\sigma q +d}} \, dy\right)^\frac pq dx \\
	& =\squaredGreek{\beta.a}+\squaredGreek{\beta.b}+\squaredGreek{\beta.c}.	
\end{align}
{In case $q=\infty$, the $L^q$ norm in $y$ is changed by an $L^\infty$ norm as usual.}

First we study the term  \squaredGreek{\beta.a}. Breaking the domain of integration into Whitney cubes, we need to use a slight variation of the covering: If $S$ has a neighbor in $\mathcal{W}_3$, we say that $S\in \mathcal{W}_3'$ and, in case $S\in \mathcal{W}_3'\setminus \mathcal{W}_3$, then we define $S^*$ to be the symmetrized of a convenient neighbor. We add and subtract the evaluation at $y$ of the approximating polynomial at the symmetrized cube $S^*$. We get that
\begin{align}\label{eqBreakBetaA}
\squaredGreek{\beta.a}
	& {{ =}} \sum_{Q\in \mathcal{W}_1} \int_Q \left( \sum_{S\in \mathcal{W}_3'} \int_{S} \frac{|\sum_{P\in \mathcal{W}_3:P\cap2S\neq \emptyset} D^{\alpha-\beta}\psi_P(y) D^\beta  P^k_{P^*} f(y) |^q}{|x-y|^{\sigma q +d}} \, dy\right)^\frac pq dx \\
\nonumber	& \lesssim \sum_{Q\in \mathcal{W}_1} \int_Q \left( \sum_{S\in \mathcal{W}_3'} \int_{S} \frac{|\sum_{P\in \mathcal{W}_3:P\cap2S\neq \emptyset} D^{\alpha-\beta}\psi_P(y) (D^\beta  P^k_{P^*} f(y) - D^\beta  P^k_{S^*} f(y) ) |^q}{|x-y|^{\sigma q +d}} \, dy\right)^\frac pq dx \\
\nonumber	& \quad {{+}}  \sum_{Q\in \mathcal{W}_1} \int_Q \left( \sum_{S\in \mathcal{W}_3'} \int_{S} \frac{|\sum_{P\in \mathcal{W}_3:P\cap2S\neq \emptyset} D^{\alpha-\beta}\psi_P(y) D^\beta  P^k_{S^*} f(y) |^q}{|x-y|^{\sigma q +d}} \, dy\right)^\frac pq dx = \squaredGreek{\beta.a.1}+\squaredGreek{\beta.a.2}
\end{align}

For the main part, \squaredGreek{\beta.a.1}, we take absolute values and we use that $|D^{\alpha-\beta}\psi_P(y)|\lesssim \ell(S)^{-|\alpha-\beta|}$. Moreover we develop the telescopic summation \rf{eqP3} along an admissible chain connecting $P^*$ and $S^*$:
\begin{align}
\squaredGreek{\beta.a.1}
\nonumber	& \lesssim \sum_{Q\in \mathcal{W}_1} \int_Q \left( \sum_{S\in \mathcal{W}_3'} \ell(S)^{-|\alpha-\beta|q}\sum_{P\cap2S\neq \emptyset}  \frac{\norm{D^\beta  P^k_{P^*} f - D^\beta  P^k_{S^*} f}_{L^q(S)}^q }{\Dist(Q,S)^{\sigma q +d}}\right)^\frac pq dx \\
\nonumber	& \lesssim \sum_{Q\in \mathcal{W}_1} \int_Q \left( \sum_{S\in \mathcal{W}_3'} \sum_{P\cap2S\neq \emptyset} \sum_{L\in [P^*,S^*]} \frac{\ell(S^*)^d \Dist(L,S^*)^{|\alpha-\beta|q}\norm{\nabla^k f - (\nabla^k f)_L}_{L^q(L)}^q.}{\ell(S)^{|\alpha-\beta|q} \ell(L)^d \Dist(Q,S)^{\sigma q +d}}\right)^\frac pq dx .
\end{align}
Note that since the cubes $2S\cap P\neq \emptyset$, they have comparable size and $\Dist(S^*,P^*)\approx \ell(S)$ by \rf{eqLongDistanceInvariant}. Thus, combining \rf{eqLengthDistance} and \rf{eqAdmissible1}, it is clear that all the elements $L\in [P^*,S^*]$ have comparable size and $\Dist(L,S^*)\approx\ell(S)$. Moreover, by \rf{eqLongDistanceInvariant}, it follows that $\Dist(Q,S)\approx \Dist(Q,S^*)\approx \Dist(Q,L)$, leaving
\begin{align}
\squaredGreek{\beta.a.1}
\nonumber	& \lesssim \sum_{Q\in \mathcal{W}_1} \int_Q \left( \sum_{S\in \mathcal{W}_3'} \sum_{P\cap2S\neq \emptyset} \sum_{L\in [P^*,S^*]} \frac{\norm{\nabla^k f - (\nabla^k f)_L}_{L^q(L)}^q.}{\Dist(Q,L)^{\sigma q +d}}\right)^\frac pq dx .
\end{align}
To complete the reduction, note that for every $L\in\mathcal{W}_1$ the number of candidates $S\in  \mathcal{W}_3'$ and $P\cap2S\neq \emptyset$ such that $L\in [S^*,P^*]$ is  uniformly bounded  by a dimensional constant. Therefore, we can use Lemma \ref{lemControlTotal} {below} to get
\begin{align}\label{eqBetaA1Bounded}
\squaredGreek{\beta.a.1}
	& \lesssim \sum_{Q\in \mathcal{W}_1}  \left(  \sum_{L\in\mathcal{W}_1} \frac{\norm{\nabla^k f - (\nabla^k f)_L}_{L^q(L)}^q.}{\Dist(Q,L)^{\sigma q +d}}\right)^\frac pq \ell(Q)^d \lesssim  C \norm{\nabla^k f}_{\dot A^\sigma_{p,q}(\Omega)}^p.
\end{align}
If $q=\infty$ the same can be obtained by trivial modifications.

On the other hand, we define $\mathcal{W}_4:=\{S\in \mathcal{W}_3: \mbox{ all the neighbors of $S$ are in $\mathcal{W}_3$}\}$.  {Given} $y\in S$ for $S\in \mathcal{W}_4$, we have that $\sum_{P\in \mathcal{W}_3}D^{\alpha-\beta}\psi_P (y) =0$. We get that
$$\squaredGreek{\beta.a.2}\leq 
\sum_{Q\in \mathcal{W}_1} \int_Q \left( \sum_{S\in \mathcal{W}_3'\setminus \mathcal{W}_4} \int_{S} \frac{|\sum_{P\cap2S\neq \emptyset} D^{\alpha-\beta}\psi_P(y) D^\beta  P^k_{S^*} f(y) |^q}{|x-y|^{\sigma q +d}} \, dy\right)^\frac pq dx.$$
{If} $S\in \mathcal{W}_3'\setminus \mathcal{W}_4$, then $\ell(S)\approx \ell_0$ and thus, $\norm{D^{\alpha-\beta}\psi_P}_\infty\approx \ell_0^{-|\alpha-\beta|}${. Summing} up, 
$$\squaredGreek{\beta.a.2}{ \lesssim} 
\sum_{Q\in \mathcal{W}_1} \left( \sum_{S\in \mathcal{W}_3'\setminus \mathcal{W}_4}  \frac{{  \ell_0^{-|\alpha-\beta| q+d}}\norm{D^\beta  P^k_{S^*} f}_{L^\infty(S)}^q}{\Dist(Q,S)^{\sigma q +d}} \right)^\frac pq  {\ell(Q)^d} .$$

{
In case $p\leq q\leq \infty$, the term can be controlled using the subadditivity of the sum, the equivalence of norms on polynomials and \rf{eqMaximalAllOver}:
$$\squaredGreek{\beta.a.2}\lesssim \ell_0^{-|\alpha-\beta|p+d \frac pq - d}
\sum_{Q\in \mathcal{W}_1}   \sum_{S\in \mathcal{W}_3'\setminus \mathcal{W}_4}  \frac{\norm{D^\beta  P^k_{S^*} f}_{L^p(S^*)}^p \ell(Q)^d}{\Dist(Q,S)^{\sigma p +d\frac pq}}\lesssim_{\ell_0} \sum_{S\in \mathcal{W}_3'\setminus \mathcal{W}_4}  \norm{D^\beta  P^k_{S^*} f}_{L^p(S^*)}^p.$$
}
For $S\in \mathcal{W}_3'\setminus\mathcal{W}_4$, { since $D^\beta  P^k_{S^*} f=  P^{k-|\beta|}_{S^*} D^\beta f$,} by \rf{eqP1Bis} we have
\begin{equation}\label{eqPolyBound}
\norm{D^\beta  P^k_{S^*} f}_{L^p(S^*)} \lesssim_{\ell_0} \norm{D^{\beta} f}_{W^{k-|\beta|,p}(S^*)} .
\end{equation}
Thus, 
\begin{equation}\label{eqBetaA2Bounded}
\squaredGreek{\beta.a.2}  \lesssim_{ \ell_0}  \norm{f}_{W^{k,p}(\Omega)}^p.
\end{equation}

{
On the other hand, in case $q<p<\infty$,  using the equivalence of norms on polynomials and \rf{eqPolyBound} we can write
$$\squaredGreek{\beta.a.2} \lesssim_{\ell_0} 
\sum_{Q\in \mathcal{W}_1} {\ell(Q)^d} \left( \sum_{S\in \mathcal{W}_1}  \norm{D^{\beta} f}_{W^{k-|\beta|,1}(S)}^q \frac{ \ell(S)^{\sigma q+d - dq } }{\Dist(Q,S)^{\sigma q +d}} \right)^\frac pq .$$
By Lemma \ref{lemNormA} below we get
\begin{align*}
\squaredGreek{\beta.a.2} 
	& \lesssim_{\ell_0} \norm{f}_{W^{k,p}(\Omega)}^p.
\end{align*}
}
By \rf{eqBreakBetaA}, \rf{eqBetaA1Bounded} and \rf{eqBetaA2Bounded}, we get
\begin{equation}\label{eqBetaABounded}
\squaredGreek{\beta.a}\lesssim \norm{f}_{A^s_{p,q}(\Omega)}^p.
\end{equation}

Next we repeat the argument for \squaredGreek{\beta.b}. This case is simpler, because we can use \rf{eqMaximalFar} to obtain
\begin{align*}
\squaredGreek{\beta.b}
	& =   \int_{\Omega^c}\left|\sum_{P\in\mathcal{W}_3} D^{\alpha-\beta}\psi_P(x) D^\beta  P^k_{P^*} f(x) \right|^p\left(\int_{\Omega} \frac{1}{|x-y|^{\sigma q +d}} \, dy\right)^\frac pq dx\\
	& \lesssim \sum_{Q\in \mathcal{W}_3'}  \ell(Q)^{-\sigma p} \int_{Q}\left|\sum_{P\in \mathcal{W}_3: P\cap 2Q\neq \emptyset} D^{\alpha-\beta}\psi_P(x) D^\beta  P^k_{P^*} f(x) \right|^p\,dx.
\end{align*}
{If $q=\infty$ we can obtain the same estimate.}

As before, we add and subtract a constant for every $x$ in the integration range to obtain
\begin{align}\label{eqBreakBetaB}
\squaredGreek{\beta.b}
	& \lesssim \sum_{Q\in \mathcal{W}_3'}  \ell(Q)^{-\sigma p-|\alpha-\beta|p}\sum_{P\in \mathcal{W}_3: P\cap 2Q\neq \emptyset}  \int_{Q}  \left|D^\beta  P^k_{P^*} f(x)-D^\beta  P^k_{Q^*} f(x)\right|^p\,dx \\
\nonumber	& \quad + \sum_{Q\in \mathcal{W}_3'\setminus \mathcal{W}_4}  \ell(Q)^{-\sigma p-|\alpha-\beta|p}\int_{Q}  \left|D^\beta  P^k_{Q^*} f(x) \right|^p\,dx 	
		= \squaredGreek{\beta.b.1}+ \squaredGreek{\beta.b.2}.
\end{align}

In the first term, we use again \rf{eqP3} and the fact that $\ell(P)\approx\ell(Q)\approx\ell(L) \approx \Dist(Q,L)$ for every $2Q\cap P\neq\emptyset$ and $L\in [Q^*,P^*]$:
\begin{align*}
\squaredGreek{\beta.b.1}
	& \lesssim \sum_{L\in \mathcal{W}_1}  \ell(L)^{-\sigma p}  \norm{\nabla^k f - (\nabla^k f)_L}_{L^p(L)}^p,
\end{align*}
Note that by Jensen's inequality we have that
\begin{align}\label{eqHisYokeIsEasy}
\sum_{L\in \mathcal{W}_1}  \ell(L)^{-\sigma p}  \norm{\nabla^k f - (\nabla^k f)_L}_{L^p(L)}^p 
	& = \sum_{L\in \mathcal{W}_1}  \int_L\left(\fint_L \frac{\nabla^k f(x) - \nabla^k f(\xi)}{\ell(L)^\sigma  } \, d\xi \right)^p dx	\\
\nonumber 	& \lesssim \sum_{L\in \mathcal{W}_1}  \int_L\left(\int_L \frac{|\nabla^k f(x) - \nabla^k f(\xi)|^q}{\ell(L)^{\sigma q+d}  } \, d\xi \right)^\frac pq dx\lesssim \norm{\nabla^k f}_{\dot A^\sigma_{p,q}(\Omega)}^p.
\end{align}
Thus,
\begin{equation}\label{eqBetaB1Bounded}
\squaredGreek{\beta.b.1}\lesssim \norm{\nabla^k f}_{\dot A^\sigma_{p,q}(\Omega)}^p.
\end{equation}

On the other hand,  we need to control the term
\begin{align*}
\squaredGreek{\beta.b.2}
	& \approx   \ell_0^{-\sigma p-|\alpha-\beta|p} \sum_{Q\in \mathcal{W}_3'\setminus \mathcal{W}_4} \norm{ D^\beta  P^k_{Q^*} f}_{L^p(Q)}^p {\lesssim_{\ell_0}}  \sum_{Q\in \mathcal{W}_1:\ell(Q)=\ell_0} \norm{ D^\beta  P^k_{Q} f}_{L^p(Q)}^p.
\end{align*}
By \rf{eqP1} we have that
\begin{align*}
\squaredGreek{\beta.b.2}
	& \lesssim_{\ell_0} \norm{f}_{W^{k,p}(\Omega)}^p.
\end{align*}
Combining this with \rf{eqBreakBetaB} and \rf{eqBetaB1Bounded}, we get
\begin{align}\label{eqBetaBBounded}
\squaredGreek{\beta.b}
	& \lesssim \norm{f}_{A^s_{p,q}(\Omega)}^p.
\end{align}

Finally  we need to deal with the term 
\begin{equation}\label{eqBetaC}
\squaredGreek{\beta.c}	= \int_{\Omega^c} \left(\int_{\Omega^c} \frac{|\sum_{P\in\mathcal{W}_3} \left((D^{\alpha-\beta}\psi_P D^\beta  P^k_{P^*} f)(x) -(D^{\alpha-\beta}\psi_P D^\beta  P^k_{P^*} f)(y)\right)|^q}{|x-y|^{\sigma q +d}} \, dy\right)^\frac pq dx .
\end{equation}
Here we will use the previous techniques but some additional tools have to be used to tackle the case $\dist(x,y) << \dist(x,\partial\Omega)$, so we separate the integration regions with this idea in mind. { We get}
\begin{align}\label{eqBreakBetaC}
\squaredGreek{\beta.c}	
\nonumber	& \leq \sum_{Q\in \mathcal{W}_4} \int_{Q} \left(\int_{B\left(x,\frac{\ell(Q)}{10}\right)} \frac{\left|\sum_{P} \left( (D^{\alpha-\beta}\psi_P D^\beta  P^k_{P^*}f)(x)-(D^{\alpha-\beta}\psi_P D^\beta  P^k_{P^*}f)(y)\right)\right|^q}{|x-y|^{\sigma q +d}} \, dy\right)^\frac pq dx\\
\nonumber	& \quad + \sum_{Q\in \mathcal{W}_2\setminus\mathcal{W}_4} \int_{Q} \left(\int_{B\left(x,\frac{\ell(Q)}{10}\right)}  \frac{\left|\sum_{P}  (D^{\alpha-\beta}\psi_P D^\beta  P^k_{P^*}f)(x)-(D^{\alpha-\beta}\psi_P D^\beta  P^k_{P^*}f)(y)\right|^q}{|x-y|^{\sigma q +d}} \, dy\right)^\frac pq dx\\
\nonumber	& \quad + \sum_{Q\in \mathcal{W}_2} \int_{Q} \left(\int_{\Omega^c\setminus B\left(x,\frac{\ell(Q)}{10}\right)}  \frac{\left|\sum_{P}  (D^{\alpha-\beta}\psi_P D^\beta  P^k_{P^*}f)(x)-(D^{\alpha-\beta}\psi_P D^\beta  P^k_{P^*}f)(y)\right|^q}{|x-y|^{\sigma q +d}} \, dy\right)^\frac pq dx\\
	& =: \squaredGreek{\beta.c.1}+\squaredGreek{\beta.c.2}+\squaredGreek{\beta.c.3}.
\end{align}

If $x\in Q \in \mathcal{W}_4$ and  $y\in  B(x,{\ell(Q)}/{10})$, then we can use the fact that  
$$\sum_{P\in\mathcal{W}_3} D^{\alpha-\beta}\psi_P(y) = \sum_{P\in\mathcal{W}_3} D^{\alpha-\beta}\psi_P(x) = 0,$$
and we can plug in constants that depend on $x$ or on $y$. We will bound the numerator of the first term in \rf{eqBreakBetaC} by
\begin{align}\label{eqBreakALot}
\nonumber& \left|\sum_{P\in\mathcal{W}_3} (D^{\alpha-\beta}\psi_P(x)-D^{\alpha-\beta}\psi_P(y)) (D^\beta  P^k_{P^*} f(x) )+ D^{\alpha-\beta}\psi_P(y) (D^\beta  P^k_{P^*} f(x) - D^\beta  P^k_{P^*} f(y)) \right| \\
		& \leq \left|\sum_{P\in\mathcal{W}_3} (D^{\alpha-\beta}\psi_P(x)-D^{\alpha-\beta}\psi_P(y)) (D^\beta  P^k_{P^*} f(x)-D^\beta  P^k_{Q^*} f(x))\right|\\
\nonumber & + \left|  \sum_{P\in\mathcal{W}_3} D^{\alpha-\beta}\psi_P(y) \left((D^\beta  P^k_{P^*} f-D^\beta  P^k_{Q^*} f)(x) - (D^\beta  P^k_{P^*} f-D^\beta  P^k_{Q^*} f)(y)\right) \right| .
\end{align}

\begin{align}\label{eqBreakBetaC2}
\squaredGreek{\beta.c.1}	
\nonumber	& \lesssim  \sum_{Q\in \mathcal{W}_4} \int_{Q} \left(\int_{B\left(x,\frac{\ell(Q)}{10}\right)} \sum_{\substack{P\in\mathcal{W}_3\\ P\cap 2Q \neq \emptyset}} \norm{\nabla D^{\alpha-\beta}\psi_P}_\infty^q \frac{\left|D^\beta  P^k_{P^*} f(x)-D^\beta  P^k_{Q^*} f(x)\right|^q}{|x-y|^{(\sigma-1) q +d}} \, dy\right)^\frac pq dx \\
\nonumber	& \quad + \sum_{Q\in \mathcal{W}_4} \int_{Q} \left(\int_{B\left(x,\frac{\ell(Q)}{10}\right)} \sum_{\substack{P\in\mathcal{W}_3\\ P\cap 2Q \neq \emptyset}} \norm{D^{\alpha-\beta}\psi_P}_\infty^q \frac{\norm{\nabla(D^\beta  P^k_{P^*} f-D^\beta  P^k_{Q^*} f)}_{L^\infty(P)}^q}{|x-y|^{(\sigma-1) q +d}} \, dy\right)^\frac pq dx \\
	& = \squaredGreek{\beta.c.1.1}+\squaredGreek{\beta.c.1.2} ,
\end{align}
{with the usual modifications when $q=\infty$.}

In the first term above, we  integrate on $y$, we use the control on the derivatives of the bump functions and we plug \rf{eqP3} in to get
\begin{align*}
\squaredGreek{\beta.c.1.1}	
	& \lesssim  \sum_{Q\in \mathcal{W}_4}  \sum_{P\in\mathcal{W}_3: P\cap 2Q \neq \emptyset} \ell(P)^{-(|\alpha-\beta|+1)p} \norm{D^\beta  P^k_{P^*} f-D^\beta  P^k_{Q^*} f}_{L^p(Q)}^p \ell(Q)^{(1-\sigma)p}  \\
	& \lesssim  \sum_{L\in \mathcal{W}_1} \ell(L)^{-(|\alpha-\beta|+1)p} \ell(L)^{(1-\sigma)p}  \frac{\ell(L)^d \ell(L)^{(|\alpha-\beta|)p}}{\ell(L)^d}\norm{\nabla^k f - (\nabla^k f)_L}_{L^p(L)} 
\end{align*}
so, by \rf{eqHisYokeIsEasy} we get
\begin{align}\label{eqBetaC11Bounded}
\squaredGreek{\beta.c.1.1}	
	& \lesssim  \sum_{L\in \mathcal{W}_1}  \ell(L)^{-\sigma p} \norm{\nabla^k f - (\nabla^k f)_L}_{L^p(L)} \lesssim  \norm{\nabla^k f}_{\dot A^\sigma_{p,q}(\Omega)}^p.
\end{align}

Note that the equivalence of norms of polynomials {implies}
$$\norm{\nabla(D^\beta  P^k_{P^*} f-D^\beta  P^k_{Q^*} f)}_{L^\infty(P)}^p\ell(P)^d \approx \norm{\nabla(D^\beta  P^k_{P^*} f-D^\beta  P^k_{Q^*} f)}_{L^p(P)}^p .$$ 
Thus, in the second term, using the same reasoning as above we get
\begin{align*}
\squaredGreek{\beta.c.1.2}	
	& \lesssim  \sum_{Q\in \mathcal{W}_4}  \sum_{P\in\mathcal{W}_3: P\cap 2Q \neq \emptyset} \ell(P)^{-|\alpha-\beta|p} \norm{\nabla(D^\beta  P^k_{P^*} f-D^\beta  P^k_{Q^*} f)}_{L^\infty(P)}^p \ell(Q)^{(1-\sigma)p+d}  \\
	& \lesssim  \sum_{L\in \mathcal{W}_1} \ell(L)^{-|\alpha-\beta| p} \frac{\ell(L)^d \ell(L)^{(|\alpha-\beta|-1)p}}{\ell(L)^d}\norm{\nabla^k f - (\nabla^k f)_L}_{L^p(L)}  \ell(L)^{(1-\sigma)p} 
\end{align*}
and we get the same case as before. By \rf{eqBreakBetaC2} and \rf{eqBetaC11Bounded} we get
\begin{equation}\label{eqBetaC1Bounded}
\squaredGreek{\beta.c.1} \lesssim  \norm{\nabla^k f}_{\dot A^\sigma_{p,q}(\Omega)}^p.
\end{equation}

{Next we} deal with the term \squaredGreek{\beta.c.2}. Whenever $x\in Q\in \mathcal{W}_2\setminus \mathcal{W}_4$ and $y\in \Omega^c \cap B(x,\ell(Q)/10)\subset \frac{11}{10}Q$, we bound the numerator in \rf{eqBreakBetaC} by the left-hand side of \rf{eqBreakALot} above:
\begin{align*}
\squaredGreek{\beta.c.2}	
\nonumber	& \lesssim  \sum_{Q\in \mathcal{W}_2\setminus\mathcal{W}_4} \int_{Q} \left(\int_{ \frac{11}{10}Q} \sum_{P\in\mathcal{W}_3: P\cap 2Q \neq \emptyset} \norm{\nabla D^{\alpha-\beta}\psi_P}_\infty^q \frac{\left|D^\beta  P^k_{P^*} f(x)\right|^q}{|x-y|^{(\sigma-1) q +d}} \, dy\right)^\frac pq dx \\
\nonumber	& \quad + \sum_{Q\in \mathcal{W}_2\setminus\mathcal{W}_4} \int_{Q} \left(\int_{ \frac{11}{10}Q} \sum_{P\in\mathcal{W}_3: P\cap 2Q \neq \emptyset} \norm{ D^{\alpha-\beta}\psi_P}_\infty^q \frac{\norm{\nabla D^\beta  P^k_{P^*} f}_{L^\infty\left(\frac{11}{10}Q\right)}^q}{|x-y|^{(\sigma-1) q +d}} \, dy\right)^\frac pq dx.
\end{align*}
{Note that only cubes $Q\in \mathcal{W}_3'$ have neighbors $P\in \mathcal{W}_3$.} Both terms are controlled by integrating on $y$ again and using the control on the derivatives of the bump functions together with \rf{eqP1Bis} and the finite overlapping of symmetrized cubes to get
\begin{align}\label{eqBetaC2Bounded}
\squaredGreek{\beta.c.2}	
\nonumber	& \lesssim  \sum_{Q\in \mathcal{W}_3'\setminus\mathcal{W}_4}  \sum_{P\in\mathcal{W}_3: P\cap 2Q \neq \emptyset}    (\ell_0^{-(|\alpha-\beta|+1)p}+\ell_0^{-(|\alpha-\beta|)p}) \norm{|D^\beta  P^k_{P^*} f|+|\nabla D^\beta  P^k_{P^*} f|}_{L^p(Q)}^p  \ell_0^{(1-\sigma)p}  \\
	& \approx \sum_{Q\in  \mathcal{W}_3'\setminus\mathcal{W}_4}  \norm{f}_{W^{k,p}(Q^*)}^p   \lesssim \norm{f}_{W^{k,p}(\Omega)}^p.
\end{align}
{If $q=\infty$, elementary modifications yield the same result.}

Let us consider the case $x\in Q \in \mathcal{W}_2$, $y\in \Omega^c\setminus B(x,\frac{\ell(Q)}{10})$. In this case we will bound the numerator in \rf{eqBreakBetaC} above by
\begin{equation}\label{eqFirstBound}
\left|\sum_{P\in\mathcal{W}_3} D^{\alpha-\beta}\psi_P(x) D^\beta  P^k_{P^*} f(x)\right| +  \left|\sum_{P\in\mathcal{W}_3}D^{\alpha-\beta}\psi_P(y) D^\beta  P^k_{P^*} f(y)\right|.
\end{equation}
We obtain
\begin{align*}
\squaredGreek{\beta.c.3}	
\nonumber	& \lesssim \sum_{Q\in \mathcal{W}_2} \int_{Q} \left(\sum_{S\in\mathcal{W}_3'}  \int_{S}\frac{\left| \sum_{P\in\mathcal{W}_3: P\cap 2S \neq \emptyset} D^{\alpha-\beta}\psi_P(y) D^\beta  P^k_{P^*} f(y) \right|^q}{\Dist(Q,S)^{\sigma q +d}} \, dy\right)^\frac pq dx \\
\nonumber	&  \quad +\sum_{Q\in \mathcal{W}_3'} \int_{Q} \left(\sum_{S\in\mathcal{W}_2}  \int_{S} \frac{\left|\sum_{P\in\mathcal{W}_3: P\cap 2Q \neq \emptyset} D^{\alpha-\beta}\psi_P(x) D^\beta  P^k_{P^*} f(x)\right|^q }{\Dist(Q,S)^{\sigma q +d}} \, dy\right)^\frac pq dx \\
	& = \squaredGreek{\beta.c.3.1}+\squaredGreek{\beta.c.3.2}.
\end{align*}
{Now, \squaredGreek{\beta.c.3.1} is bounded as \squaredGreek{\beta.a}, and \squaredGreek{\beta.c.3.2} is bounded as \squaredGreek{\beta.b} without much change.} Combining these estimates we obtain
\begin{equation}\label{eqBetaC3Bounded}
\squaredGreek{\beta.c.3} \lesssim \norm{f}_{A^s_{p,q}(\Omega)}^p.
\end{equation}

Combining \rf{eqBreakBetaC}, \rf{eqBetaC1Bounded},  \rf{eqBetaC2Bounded} and \rf{eqBetaC3Bounded} we have
\begin{equation*}
\squaredGreek{\beta.c} \lesssim \norm{f}_{A^s_{p,q}(\Omega)}^p, 
\end{equation*}
which combined with \rf{eqBreakBeta}, \rf{eqBetaABounded} and \rf{eqBetaBBounded}, leads to
\begin{equation*}
\squaredGreek{\beta} \lesssim \norm{f}_{A^s_{p,q}(\Omega)}^p
\end{equation*}
and the theorem follows.
\end{proof}

{
It remains to proof a couple of technical lemmata used during the proof of the  boundedness of the extension operator. }

\begin{lemma}\label{lemControlTotal}
Let  $d\geq 1$ be a natural number, let $0<\sigma<1$, let $1\leq p<\infty$, {and let  $1\leq q \leq \infty$} with $\sigma >\frac dp-\frac dq$. There exists a constant $C$ such that for every  $f\in L^p(\Omega)$, 
$$\sum_{Q\in \mathcal{W}_1 \cup \mathcal{W}_2}\left(\sum_{L\in \mathcal{W}_1} \int_L\frac{| f(y)-f_L|^q }{\Dist(Q,L)^{\sigma q+ d}} dy \right)^\frac pq \ell(Q)^d \leq C \norm{f}_{\dot A^\sigma_{p,q}(\Omega)}^p,$$
{with the usual modifications when $q=\infty$.}
\end{lemma}
\begin{proof}
Let us write
$$\squared{A}:=\sum_{Q\in \mathcal{W}_1  \cup \mathcal{W}_2}\left(\sum_{L\in \mathcal{W}_1} \int_L\frac{| f(y)-f_L|^q }{\Dist(Q,L)^{\sigma q+ d}} dy \right)^\frac pq \ell(Q)^d.$$

Consider first the case $1\leq p=q<\infty$. In this case we can change  the order of summation, and it is enough to bound
$$\squared{A}=\sum_{L\in \mathcal{W}_1} \int_L | f(y)-f_L|^p dy \sum_{Q\in \mathcal{W}_1  \cup \mathcal{W}_2} \frac{\ell(Q)^d}{\Dist(Q,L)^{\sigma p+ d}} .$$
The sum in $Q$ is controlled by \rf{eqMaximalAllOver}, and using Jensen's inequality we obtain
$$\squared{A}\lesssim \sum_{L\in \mathcal{W}_1} \int_L | f(y)-f_L|^p dy \frac{1}{\ell(L)^{\sigma p}} \lesssim \sum_{L\in \mathcal{W}_1} \int_L \int_L \frac{| f(y)-f(\xi)|^p}{\ell(L)^{\sigma p+d}} \, d\xi \, dy\lesssim_d \norm{f}_{\dot A^\sigma_{p,p}(\Omega)}^p  .$$

The case $p<q<\infty$ is gotten by a slight modification, and using the fact that $x\mapsto x^{\frac pq}$ is sub-aditive. Indeed, note that given $y\in L\in \mathcal{W}_1$, we {obtain} that 
$$\squared{A}\leq \sum_{Q\in \mathcal{W}_1 \cup \mathcal{W}_2}\sum_{L\in \mathcal{W}_1} \left(\int_L\frac{| f(y)-f_L|^q }{\Dist(Q,L)^{\sigma q+ d}} dy \right)^\frac pq \ell(Q)^d.$$
Again we can change the order of summation. Now we will use that 
$$\sum_{Q\in\mathcal{W}_1\cup \mathcal{W}_2}\frac{\ell(Q)^d}{\Dist(Q,L)^{\sigma p + \frac{dp}{q}}}\lesssim \ell(L)^{d-\sigma p - \frac{dp}{q}}$$
by \rf{eqMaximalAllOver}, since $\sigma p + \frac{dp}{q}> d$ by assumption. 
On the other hand, using Minkowski and Jensen's inequalities we have that 
\begin{align*}
\left(\int_L| f(y)-f_L|^q \, dy \right)^\frac pq
	& \leq \left(\fint_L \left( \int_L | f(y)-f(\xi)|^q dy\right)^\frac1q d\xi \right)^p \leq \fint_L \left( \int_L | f(y)-f(\xi)|^q dy\right)^\frac pq d\xi  .
\end{align*}
Thus, 
$$\squared{A}\leq \sum_{L\in \mathcal{W}_1} \fint_L \left( \int_L | f(y)-f(\xi)|^q dy\right)^\frac pq d\xi \,  \ell(L)^{d-\sigma p - \frac{dp}{q}} = \sum_{L\in \mathcal{W}_1} \int_L \left( \int_L \frac{| f(y)-f(\xi)|^q}{\ell(L)^{\sigma q + d}} dy\right)^\frac pq d\xi ,$$
and the lemma follows. {The case $p<q=\infty$ is obtained by trivial modifications using $\sigma p>d$.}

For the case $p>q>1$ we will follow a longer reasoning, using Lemma \ref{lemNormA} below. As above, by Minkowski's inequality we get
\begin{align*}
\left(\int_L| f(y)-f_L|^q \, dy \right)^\frac1q
	& \leq \fint_L \left( \int_L | f(y)-f(\xi)|^q dy\right)^\frac1q d\xi \lesssim \int_L \left( \int_L \frac{| f(y)-f(\xi)|^q}{|{y}-\xi|^{\sigma q +d} }dy\right)^\frac1q \frac{ d\xi}{ \ell(L)^{d-\sigma -\frac dq}}.
\end{align*}
Thus, writing $D^\sigma_q f(\xi):=\left( \int_\Omega \frac{| f(y)-f(\xi)|^q}{|{y}-\xi|^{\sigma q +d} }dy\right)^\frac1q$, we get
\begin{align}\label{eqAStar}
\squared{A}
	& \lesssim \sum_{Q\in \mathcal{W}_1  \cup \mathcal{W}_2}\left( \sum_{L\in \mathcal{W}_1} \left(\int_L D^\sigma_q f (\xi)  d\xi\right)^q \left(\frac{ \ell(L)^{\sigma +\frac dq-d}}{\Dist(Q,L)^{\sigma + \frac dq}} \right)^q \right)^\frac pq \ell(Q)^d.
\end{align}

{By Lemma \ref{lemNormA}, we get
\begin{align*}
\squared{A}^\frac1p
	& \lesssim \norm{D^\sigma_q f}_{L^p(\Omega)}=\norm{f}_{\dot A^\sigma_{p,q}(\Omega)},
\end{align*}
and the lemma follows.} \end{proof}

{\begin{lemma}\label{lemNormA}
Let  $d\geq 1$ be a natural number, let $0<\sigma<1$, let $1\leq q<p<\infty$. There exists a constant $C$ such that for every  $h\in L^p(\Omega)$, 
$$\left(\sum_{Q\in \mathcal{W}_1 \cup  \mathcal{W}_2}\left( \sum_{L\in \mathcal{W}_1} \norm{h}_{L^1(L)}^q  \left(\frac{ \ell(L)^{\sigma +\frac dq-d}}{\Dist(Q,L)^{\sigma + \frac dq}} \right)^q \right)^\frac pq \ell(Q)^d \right)^\frac1p \lesssim \norm{h}_{L^p}.$$
\end{lemma}}
\begin{proof}
We will use duality and the boundedness of the maximal operator in $L^{\frac{p}{\widetilde{p}}}$ for $q<\widetilde{p}<p$.  First consider $f(x,y) = \sum_{Q\in \mathcal{W}_1} \sum_{L\in \mathcal{W}_1} \chi_Q(x) \chi_L (y) \norm{h}_{L^1(L)} \frac{ \ell(L)^{\sigma -d}}{\Dist(Q,L)^{\sigma + \frac dq}}$. Then by duality we get
\begin{align*}
\squared{B}
	& := \norm{f}_{L^p_x(L^q_y)} = \sup_{\norm{g}_{L^{p'}_x\left(L^{q'}_y\right)} \leq 1}  \int_{\Omega} \int_\Omega f(x,y) g(x,y) \, dy\, dx\\
	& =  \sup_{\norm{g}_{L^{p'}_x\left(L^{q'}_y\right)} \leq 1} \sum_{Q\in \mathcal{W}_1} \sum_{L\in \mathcal{W}_1} \norm{h}_{L^1(L)} \frac{ \ell(L)^{\sigma -d}}{\Dist(Q,L)^{\sigma + \frac dq}}   \int_Q\int_L g(x,y) \, dy\, dx.
\end{align*}

Consider a function $g$ such that $\norm{g}_{L^{p'}_x\left(L^{q'}_y\right)}\leq 1$. In the sum above,  for every $Q$ and ${L}$ appearing in the sum we consider a chain $[Q,{L}]$ with central cube $R=Q_{{L}}$. Then, using \rf{eqAdmissible2} and reordering we get
\begin{align*}
\squared{B}_g
	& :=  \sum_{Q\in \mathcal{W}_1} \sum_{L\in \mathcal{W}_1} \norm{h}_{L^1(L)} \frac{ \ell(L)^{\sigma -d}}{\ell(Q_{{L}})^{\sigma + \frac dq}}   \int_Q\int_L g(x,y) \, dy\, dx \\
	& \leq \sum_{Q\in \mathcal{W}_1} \sum_{R: Q\in \SH(R)} \sum_{L\in \SH(R)} \norm{h}_{L^1(L)} \frac{ \ell(L)^{\sigma -d}}{\ell(R)^{\sigma + \frac dq}}   \int_Q\int_L g(x,y) \, dy\, dx \\
	& \leq \sum_{R\in \mathcal{W}_1}  \frac{ 1}{\ell(R)^{\sigma + \frac dq}} \sum_{Q\in \SH(R)} \int_Q \sum_{L\in \SH(R)} \norm{h}_{L^1(L)}  \ell(L)^{\sigma -d}  \left(\int_L g(x,y)^{q'} \, dy\right) ^\frac1{q'}   \ell(L)^\frac{d}{q} dx.
\end{align*}
Next we  apply the generalized H\"older inequality  to get
\begin{align*}
\squared{B}_g
	& \leq \sum_{R\in \mathcal{W}_1}  \frac{ 1}{\ell(R)^{\sigma + \frac dq}} \sum_{Q\in \SH(R)} 
	\int_Q \left(\sum_{L\in \SH(R)} \int_L g(x,y)^{q'} \, dy\right)^\frac1{q'} dx\\
	& \quad \left(\sum_{L\in \SH(R)} \left( \frac{\norm{ h}_{L^1(L)}}{\ell(L)^{d\left(1-\frac1{\widetilde{p}}\right)}}\right)^{\widetilde{p}}\right)^\frac1{\widetilde{p}}
	\left(\sum_{L\in \SH(R)} \left(\ell(L)^{\sigma+\frac{d}{q}-\frac{d}{\widetilde{p}}}\right)^\frac{\widetilde{p}q}{\widetilde{p}-q}\right)^\frac{\widetilde{p}-q}{\widetilde{p}q}.
\end{align*}

Let us bound the three terms. First, by increasing the domain of integration, we have that
$$\left(\sum_{L\in \SH(R)} \int_L g(x,y)^{q'} \, dy\right)^\frac1{q'} \leq \norm{g(x,\cdot)}_{L^{q'}(\Omega)} =: G(x).$$
Secondly, by the H\"older inequality and \rf{eqMaximalGuay} we get
$$\sum_{L\in \SH(R)} \left( \frac{\norm{ h}_{L^1(L)}}{\ell(L)^{d\left(1-\frac1{\widetilde{p}}\right)}}\right)^{\widetilde{p}}\leq \sum_{L\in \SH(R)} \int_L |h(x)|^{\widetilde{p}} \, dx \leq \ell(R)^d \inf_{\zeta \in R} M\left(|h|^{\widetilde{p}}\right)(\zeta).$$
Finally, by \rf{eqMaximalAllOver} we get
$$\left(\sum_{L\in \SH(R)} \left(\ell(L)^{\sigma+\frac{d}{q}-\frac{d}{\widetilde{p}}}\right)^\frac{\widetilde{p}q}{\widetilde{p}-q}\right)^\frac{\widetilde{p}-q}{\widetilde{p}q}
	= \left(\sum_{L\in \SH(R)} \left(\ell(L)^{\frac{\sigma\widetilde{p}q}{\widetilde{p}-q} +d}\right)\right)^\frac{\widetilde{p}-q}{\widetilde{p}q} 	
	\lesssim \ell(R)^{\sigma+\frac{d}{q}-\frac{d}{\widetilde{p}}} .$$
All together, we have gotten
\begin{align*}
\squared{B}_g
	& \leq \sum_{R\in \mathcal{W}_1}  \frac{ 1}{\ell(R)^{\sigma + \frac dq}} \sum_{Q\in \SH(R)} 
	\int_Q G(x) \, dx 
	\left( \ell(R)^d \inf_{\zeta \in R} M\left(|h|^{\widetilde{p}}\right)(\zeta)\right)^\frac1{\widetilde{p}}
	\ell(R)^{\sigma+\frac{d}{q}-\frac{d}{\widetilde{p}}}.
\end{align*}

Using \rf{eqMaximalGuay} again we get that $ \sum_{Q\in \SH(R)} \int_Q G(x) \, dx \lesssim \int_R MG(\zeta)\, d\zeta$ and, computing we get
\begin{align*}
\squared{B}_g
	& \leq \sum_{R\in \mathcal{W}_1}   \int_R MG(\zeta) \left( M\left(|h|^{\widetilde{p}}\right)(\zeta)\right)^\frac1{\widetilde{p}}\, d\zeta
		\leq \norm{ MG}_{L^{p'}(\Omega)}  \norm{ M\left(|h|^{\widetilde{p}}\right)}_{L^\frac{p}{\widetilde{p}}(\Omega)}^\frac1{\widetilde{p}}\\
	& \lesssim \norm{G}_{L^{p'}(\Omega)}  \norm{|h|^{\widetilde{p}}}_{L^\frac{p}{\widetilde{p}}(\Omega)}^\frac1{\widetilde{p}}
		=  \norm{g}_{L^{p'}_x(\Omega)(L^{q'}_y(\Omega))}  \norm{h}_{L^p(\Omega)} \leq   \norm{h}_{L^p(\Omega)}.
\end{align*}
{Finally the case $p>q=1$ works with trivial modifications, noting that $q'=\infty$.}

On the other hand, for $f(x,y) = \sum_{Q\in \mathcal{W}_2} \sum_{L\in \mathcal{W}_1} \chi_Q(x) \chi_L (y) \norm{h}_{L^1(L)} \frac{ \ell(L)^{\sigma -d}}{\Dist(Q,L)^{\sigma + \frac dq}}$, again by duality we get
\begin{align*}
\squared{C}
	& := \norm{f}_{L^p_x(L^q_y)}  =  \sup_{\norm{g}_{L^{p'}_x\left(L^{q'}_y\right)} \leq 1} \sum_{Q\in \mathcal{W}_2} \sum_{L\in \mathcal{W}_1} \norm{h}_{L^1(L)} \frac{ \ell(L)^{\sigma -d}}{\Dist(Q,L)^{\sigma + \frac dq}}   \int_Q\int_L g(x,y) \, dy\, dx.
\end{align*}

The estimate
$$\squared{C1}:=\sum_{Q\in \mathcal{W}_3} \sum_{L\in \mathcal{W}_1} \norm{h}_{L^1(L)} \frac{ \ell(L)^{\sigma -d}}{\Dist(Q,L)^{\sigma + \frac dq}}   \int_Q\int_L g(x,y) \, dy\, dx
 \leq C \norm{h}_{L^p(\Omega)}.$$
is obtained {changing $[Q,L]$ by $[Q^*,L]$ and $Q\in \SH(R)$ by $Q^*\in \SH(R)$ in the preceding proof. }

{If $\Omega$ is unbounded the same can be done for the sum in  $Q\in \mathcal{W}_2\setminus \mathcal{W}_3$, because $Q^*$ is always defined. If, instead, $\Omega$ is bounded, then  using $\Dist(Q,L)\geq \ell(Q)$ we get}
{
\begin{align*}
\squared{C2}
	& :=\sum_{Q\in\mathcal{W}_2\setminus \mathcal{W}_3} \sum_{L\in \mathcal{W}_1} \norm{h}_{L^1(L)} \frac{ \ell(L)^{\sigma -d}}{\Dist(Q,L)^{\sigma + \frac dq}}   \int_Q\int_L g(x,y) \, dy\, dx\\
	& \leq \sum_{Q\in \mathcal{W}_2\setminus\mathcal{W}_3}  \frac{ 1}{\ell(Q)^{\sigma + \frac dq}} 
	\int_Q \left(\sum_{L\in \mathcal{W}_1} \int_L g(x,y)^{q'} \, dy\right)^\frac1{q'} dx\\
	& \quad \left(\sum_{L\in \mathcal{W}_1} \left( \frac{\norm{ h}_{L^1(L)}}{\ell(L)^{d\left(1-\frac1{p}\right)}}\right)^{p}\right)^\frac1{p}
	\left(\sum_{L\in \mathcal{W}_1} \left(\ell(L)^{\sigma+\frac{d}{q}-\frac{d}{p}}\right)^\frac{pq}{p-q}\right)^\frac{p-q}{pq}\\
	& \leq \sum_{Q\in \mathcal{W}_2\setminus\mathcal{W}_3}  \frac{ 1}{\ell(Q)^{\sigma + \frac dq}} 
	\int_Q G(x) \, dx  \norm{ h}_{L^p(\Omega)} 
	\diam(\Omega)^{\sigma+\frac{d}{q}-\frac{d}{p}}\\
	& \leq \left(\sum_{Q\in \mathcal{W}_2\setminus\mathcal{W}_3}  \frac{ 1}{\ell(Q)^{\left(\sigma + \frac dq-\frac dp\right)p}} \right)^\frac1p
	\norm{ G}_{L^{p'}(\Omega^c)}  \norm{ h}_{L^p(\Omega)} 
	\diam(\Omega)^{\sigma+\frac{d}{q}-\frac{d}{p}}.
\end{align*}
}The last sum is a convergent geometric series because $\sigma + \frac dq -\frac dp >0$ by assumption and the number of cubes of a given size is uniformly bounded by a constant which depends on $\ell_0$, that is, on the uniformity constants and the diameter of the domain.

\end{proof}

\section{Proof of the main result}\label{secMain}
\begin{proof}[Proof of Corollary \ref{coroPS}]
From \cite[Theorem 1.2]{PratsSaksman} it follows that
$$\norm{f}_{F^s_{p,q}}\approx \norm{f}_{W^{k,\max\{p,q\}}}+ \left(\int_{\R^d} \left( \int_{\R^d} \frac{|\nabla^kf(x)-\nabla^kf(y)|^q}{|x-y|^{\sigma q+d}}\, dy\right)^\frac pq  dx\right)^\frac1p.$$

However, condition $f\in W^{k,\max\{p,q\}}$ is not necessary: Given a function $g\in L^1_{loc}$ and $q\geq p$, it follows that 
$$\norm{g}_{L^q(\R^d)}\lesssim \norm{g}_{L^p(\R^d)} + \left(\int_{\R^d} \left( \int_{\R^d} \frac{|g(x)-g(y)|^q}{|x-y|^{\sigma q+d}}\, dy\right)^\frac pq  dx\right)^\frac1p$$
(see \cite[estimate (3.8)]{PratsSaksman}, which is valid for any open set). 
\end{proof}

\begin{proof}[Proof of Theorem \ref{theoEquivalent}] 
By Corollary \ref{coroPS} and Theorem \ref{theoExtension} we get
$$\norm{f}_{F^s_{p,q}(\Omega)}= \inf_{g|_\Omega \equiv f} \norm{g}_{F^s_{p,q}(\R^d)}\approx \inf_{g|_\Omega \equiv f} \norm{g}_{A^s_{p,q}(\R^d)} \leq \norm{\Lambda^k f}_{A^s_{p,q}(\R^d)} \lesssim \norm{f}_{A^s_{p,q}(\Omega)}.$$

On the other hand, again by Corollary \ref{coroPS} we get the converse
$$ \norm{f}_{A^s_{p,q}(\Omega)}\leq \inf_{g|_\Omega \equiv f} \norm{g}_{A^s_{p,q}(\R^d)} \approx \inf_{g|_\Omega \equiv f} \norm{g}_{F^s_{p,q}(\R^d)}=\norm{f}_{F^s_{p,q}(\Omega)}.$$
\end{proof}

\appendix
{\section{Endpoint cases and unbounded domains}}

The extension theorem for $k=0$ in \cite[Theorem1.4]{PratsSaksman} is only proven for bounded domains and for $1<p,q<\infty$. The cases $p=1$ or $q=1$ the proof can be run with no changes at all. Below we detail how to adapt the proof of the extension theorem for $k=0$ including the endpoint $q=\infty$ and with $\Omega$ possibly unbounded:

\begin{theorem}\label{theoEndpoint}
Let $\Omega\subset \R^d$ be a uniform domain, $1\leq p<\infty$, $1\leq q\leq \infty$ and $0<s<1$ with $s>\frac dp-\frac dq$. Then $\Lambda_0$ as described in Section \ref{secExtension} is a bounded extension operator mapping $A^s_{p,q}(\Omega)$ to $F^s_{p,q}(\R^d)$.

In particular, $A^s_{p,q}(\Omega)=F^s_{p,q}(\Omega)$ in the sense of equivalent norms.
\end{theorem}
\begin{proof}[Sketch of the proof]
Next we list the changes needed in the original proof in order to include the case of unbounded domains and the endpoint $q=\infty$. 

First of all, although we can take $\ell_0=\infty$ in \cite[Definition of $\Lambda_0$]{PratsSaksman} when $\Omega$ is unbounded, this would not satisfy \rf{eqLambda0}. Thus, we need to assume $\ell_0$ to be finite ($\ell_0=1$ would do the job).

The sketch of the proof is the same as in the original article. When $q=\infty$ we need to define
$$\circled{a}=\int_\Omega \sup_{y\in \Omega^c} \frac{|f(x)-\Lambda_0 f(y)|^p}{|x-y|^{sp}}\, dx, $$
$$\circled{b}=\int_{\Omega^c} \sup_{y\in \Omega} \frac{|\Lambda_0f(x)- f(y)|^p}{|x-y|^{sp}}\, dx,$$
and
$$\circled{c}=\int_{\Omega^c} \sup_{y\in \Omega^c} \frac{|\Lambda_0f(x)-\Lambda_0 f(y)|^p}{|x-y|^{sp}}\, dx.$$

We can bound $\circled{a}\lesssim \circled{a1}+\circled{a2}$ in the same way as in the original article.
In case $\Omega$ is unbounded, the term
 $$\circled{a1}= \sum_{Q\in\mathcal{W}_1}\int_Q  \left(\sum_{S\in \mathcal{W}_3} \frac{|f(x)-f_{S^*}|^q}{D(Q,S)^{sq+d}} \ell(S)^d  \right)^\frac pq \, dx $$
is bounded using the Jensen inequality without variation with respect to the original argument. If $q=\infty$, use that $ |f(x)-f_{S^*}|^p\leq  \fint_{S^*} |f(x)-f(\xi)|^p\, d\xi$ to obtain
$$\circled{a1}= \sum_{Q\in\mathcal{W}_1}\int_Q \sup_{S\in \mathcal{W}_3}  \frac{|f(x)-f_{S^*}|^p}{\Dist(Q,S)^{sp}} \,dx\lesssim \int_\Omega \sup_{S\in \mathcal{W}_1} \sup_{\xi\in S}  \frac{|f(x)-f(\xi)|^p}{|x-\xi|^{sp}}\,dx= \norm{f}_{\dot{A}^s_{p,\infty}(\Omega)}^p.$$
Here and in the remaining of the appendix, we write supremum for the essential supremum with respect to the Lebesgue measure.

To control 
$$\circled{a2}= \sum_{Q\in\mathcal{W}_1}\int_Q  \left(\sum_{S\in \mathcal{W}_2\setminus \mathcal{W}_4} \frac{|f(x)|^q}{D(Q,S)^{sq+d}} \ell(S)^d \right)^\frac pq \, dx$$ in the unbounded case, use that 
$$\sum_{S\in \mathcal{W}_2\setminus \mathcal{W}_4}\frac{\ell(S)^d}{\Dist(Q,S)^{sq+d}}\lesssim \int_{R^d} \frac{1}{(|x| + \ell_0)^{sq+d}}\, dx\approx \frac{1}{\ell_0^{sq}}$$
to obtain $\circled{a2}\lesssim \frac{\norm{f}_{L^p(\Omega)}^p}{\ell_0^{sp}}$.
An analogous inequality works whenever $q=\infty$.

In $\circled{b}$ a similar decomposition is used. The term
$$\circled{b1} = \sum_{Q\in \mathcal{W}_3}\ell(Q)^d \left(\sum_{S\in\mathcal{W}_1} \int_S \frac{|f_{Q^*} - f(y)|^q}{\Dist(Q,S)^{sq+d}}\, dy\right)^\frac pq  $$
is bounded without change both in the unbounded case and in case $q=\infty$, using the Minkowski inequality 
$$\norm{\sum_S\frac{\int_{Q^*}|f(\xi)-f(\cdot)|\,d\xi}{D(Q^*,S)^{s+d/q}}\chi_S}_{L^q}\leq \int_{Q^*} \norm{\sum_S\frac{|f(\xi)-f(\cdot)|}{D(Q^*,S)^{s+d/q}}\chi_S}_{L^q}\, d\xi$$
to obtain $\circled{b1}\lesssim\norm{f}_{\dot A^s_{p,q}(\Omega)}^p$.

The first term that needs a different approach is 
$$\circled{b2}=\sum_{Q\in \mathcal{W}_2\setminus \mathcal{W}_4} \ell(Q)^d\left(\int_\Omega \frac{|f(y)|^q}{\dist(y,Q)^{sq+d}}\, dy\right)^\frac pq $$
If $\Omega$ is unbounded, all cubes in $\mathcal{W}_2$ have a symmetrized cube, but if $\Omega$ is bounded this is not true. Thus, we  add an subtract $f(x)$ for $x\in Q^*$ or in a fixed $Q_0\in \mathcal{W}_1$ to get
\begin{align*}
\circled{b2}
	& \lesssim \sum_{Q\in \mathcal{W}_2\setminus \mathcal{W}_4: \exists Q^*} \int_{Q^*} \left( \int_\Omega \frac{|f(y)-f(x)|^q}{|x-y|^{sq+d}}\, dy\right)^\frac pq dx \\
	& \quad +\sum_{Q\in \mathcal{W}_2\setminus \mathcal{W}_4: \exists Q^*} \int_{Q^*} |f(x)|^pdx \left( \int_\Omega \frac{1}{\dist(y,Q)^{sq+d}}\, dy\right)^\frac pq \\
	& \quad + \sum_{Q\in \mathcal{W}_2\setminus \mathcal{W}_4: \nexists Q^*} \int_{Q_0} \left( \int_\Omega \frac{|f(y)-f(x)|^q}{\dist(y,Q)^{sq+d}}\, dy\right)^\frac pq dx \frac{\ell(Q)^d}{\ell(Q_0)^d} \\
	& \quad +\sum_{Q\in \mathcal{W}_2\setminus \mathcal{W}_4: \nexists Q^*} \int_{Q_0} |f(x)|^pdx \left( \int_\Omega \frac{1}{\dist(y,Q)^{sq+d}}\, dy\right)^\frac pq \frac{\ell(Q)^d}{\ell(Q_0)^d}= \circled{b3}+\circled{b4}+\circled{b5}+\circled{b6}.
	\end{align*}
The first term is controlled trivially by $\circled{b3}\lesssim \norm{f}_{\dot A^s_{p,q}(\Omega)}^p$. For the second, note that being $\ell(Q)\geq \ell_0$, we have that $\int_{\R^d} \frac{1}{\dist(y,Q)^{sq+d}}\, dy\lesssim \ell_0^{-sq}$, so $\circled{b4}\lesssim \norm{f}_{L^p(\Omega)}^p \ell_0^{-sp}$. If $q=\infty$ the same applies in both cases.

If $\Omega$ was unbounded, the last two terms, $\circled{b5}$ and $\circled{b6}$ would vanish. Thus, we can assume $\Omega$ to be bounded, and then we only need to deal with $q=\infty$. The first one is
$$\circled{b5}=\sum_{\substack{Q\in \mathcal{W}_2\setminus \mathcal{W}_4:\\ \nexists Q^*}} \int_{Q_0}  \sup_{y\in\Omega} \frac{|f(y)-f(x)|^p}{\dist(y,Q)^{sp}} dx \frac{\ell(Q)^d}{\ell(Q_0)^d}\leq \int_{Q_0}  \sup_{y\in\Omega} |f(y)-f(x)|^pdx \sum_{\substack{Q\in \mathcal{W}_2\setminus \mathcal{W}_4:\\ \nexists Q^*}} \frac{\ell(Q_0)^{-d}}{\ell(Q)^{sp-d}}. $$
The last sum is a convergent geometric sum as long as $\Omega$ is bounded. We get
$$\circled{b5}\lesssim \frac{\ell(Q_0)^{sp-d}}{\ell_0^{sp-d}} \norm{f}_{\dot A^s_{p,q}(\Omega)}^p.$$
To end, 
$$\circled{b6}=\sum_{Q\in \mathcal{W}_2\setminus \mathcal{W}_4: \nexists Q^*} \int_{Q_0} |f(x)|^p dx  \sup_{y\in\Omega} \frac{1}{\dist(y,Q)^{sp}} \frac{\ell(Q)^d}{\ell(Q_0)^d} = \int_{Q_0} |f(x)|^p dx  \sum_{Q\in \mathcal{W}_2\setminus \mathcal{W}_4: \nexists Q^*} \frac{\ell(Q_0)^{-d}}{\ell(Q)^{sp-d}}$$
and $\circled{b6} \lesssim \frac{\norm{f}_{L^p(\Omega)}^p}{\ell_0^{sp-d}}\ell(Q_0)^{-d} $.

The last term $\circled{c}$ can be obtained by means of all the techniques described above and we omit the details here.
\end{proof}

\begin{theorem}\label{theoShadow}
Let $\Omega$ be a uniform domain with an admissible Whitney covering $\mathcal{W}$, let $0<s<1$, $1\leq p <\infty$ and $1\leq q\leq \infty$. If $f\in L^1_{\rm loc} (\Omega)$ then
$$\norm{f}_{{\widetilde{A}}^s_{p,q}(\Omega)}:= \left( \sum_{Q\in\mathcal{W}} \int_Q \left(\int_{Sh(Q)} \frac{|f(x)-f(y)|^q}{|x-y|^{sq+d}}\, dy \right)^\frac pq dx\right)^\frac1p \approx \norm{f}_{\dot{A}^s_{p,q}(\Omega)}.$$
\end{theorem}
\begin{proof}
The case $1<p,q<\infty$ is essentially contained \cite[Lemma 4.1, Remark 4.2]{PratsSaksman}, with $\Omega$ bounded, since the restrictions $f\in L^p$, $s>\frac dp-\frac dq$ and the boundedness of $\Omega$ are never used along that proof. The same proof can be applied in the case $1\leq p,q <\infty$  with the usual modifications when the dual exponents are not finite.  It remains to show the case $q=\infty$, $1\leq p <\infty$, when the duality argument cannot be used.

In this case, 
$$\norm{f}_{\dot A^s_{p,q}(\Omega)}^p = \int_\Omega \sup_{y\in \Omega} \frac{|f(x)-f(y)|^p}{|x-y|^{sp}}\, dx$$
and 
$$\norm{f}_{\widetilde A^s_{p,q}(\Omega)}^p =\sum_{Q\in\mathcal{W}} \int_Q  \sup_{y\in \Sh(Q)} \frac{|f(x)-f(y)|^p}{|x-y|^{sp}}\, dx .$$
The fact that $\norm{f}_{\widetilde A^s_{p,q}(\Omega)}^p \leq \norm{f}_{\dot A^s_{p,q}(\Omega)}^p$ is trivial. To see the converse, it suffices to show that
$$\sum_{Q\in\mathcal{W}} \int_Q  \sup_{S\in \mathcal{W}} \sup_{y \in S} \frac{|f(x)-f(y)|^p}{\Dist(Q,S)^{sp}}\, dx \lesssim  \norm{f}_{\widetilde A^s_{p,q}(\Omega)}^p.$$
Using as usual admissible chains joining ths cubes in consideration and the triangle inequality it is enough to show that  
$$\circled{1}:=\sum_{Q\in \mathcal{W}} \int_Q \sup_{S\in \mathcal{W}} \frac{|f(x)-f_Q|^p}{\Dist(Q,S)^{sp}}\, dx\lesssim  \norm{f}_{\widetilde A^s_{p,q}(\Omega)}^p,$$
$$\circled{2}:=\sum_{Q\in \mathcal{W}} \ell(Q)^d \sup_{S\in \mathcal{W}} \frac{|f_Q-f_{Q_S}|^p}{\Dist(Q,S)^{sp}}\lesssim  \norm{f}_{\widetilde A^s_{p,q}(\Omega)}^p,$$
and
$$\circled{3}:=\sum_{Q\in \mathcal{W}} \ell(Q)^d \sup_{S\in \mathcal{W}} \sup_{y \in S} \frac{|f_{Q_S}-f(y)|^p}{\Dist(Q,S)^{sp}}\lesssim  \norm{f}_{\widetilde A^s_{p,q}(\Omega)}^p.$$

The first sum can be bounded almost immediately. Indeed, $|f(x)-f_Q|\leq \sup_{\xi\in Q} |f(x)-f(\xi)|$ and $\sup_{S}\frac{1}{\Dist(Q,S)^{sp}}\leq \ell(Q)^{-sp}$, so
$$\circled{1}\lesssim \sum_Q \int_Q  \sup_{\xi\in Q} \frac{|f(x)-f(\xi)|^p}{|x-\xi|^{sp}}\, dx\leq  \norm{f}_{\widetilde A^s_{p,q}(\Omega)}^p.$$

In the second case, we use the subchain  $[Q,Q_S]\subset [Q,S]$ to write
$$\circled{2}\leq \sum_Q  \sup_{S} \left| \sum_{P\in [Q,Q_S)}(f_P-f_{\mathcal{N}(P)})\right|^p \frac{\ell(Q)^d}{\Dist(Q,S)^{sp}},$$
where $\mathcal{N}(P)$ stands for `the next cube in the chain'. Using the H\"older inequality,  \rf{eqAscendingPath} and \rf{eqAdmissible2} we get
\begin{align*}
\circled{2}
	& \lesssim \sum_Q  \sup_{S} \sum_{P\in [Q,Q_S)} |f_P-f_{\mathcal{N}(P)}|^p \ell(P)^{-s p/2} \ell(Q_S)^{s p/2} \frac{\ell(Q)^d}{\Dist(Q,S)^{sp}}\\
	& \lesssim \sum_Q \sup_{R: Q\in \SH(R)} \sup_{S\in \SH(R)} \sum_{P\in \SH(R): Q\in \SH(P)} \fint_P\fint_{5P}|f(\xi)-f(\zeta)|^p\, d\zeta \, d\xi \frac{\ell(R)^{s p/2} \ell(Q)^d}{\ell(P)^{s p/2} \ell(R)^{sp}}.
\end{align*}
Note that the dependence on $S$ has vanished. Next we change the supremum in $R$ by a sum, and change the order of summation to find
$$\circled{2}\lesssim \sum_{P\in \mathcal{W}} \fint_P\fint_{5P} \frac{|f(\xi)-f(\zeta)|^p}{ \ell(P)^{s p/2}}\, d\zeta \, d\xi   \sum_{R: P\in \SH(P)}   \frac{1}{\ell(R)^{sp/2}}\sum_{Q\in \SH(R)}\ell(Q)^d .$$
By \rf{eqAscendingToGlory} and \rf{eqMaximalAllOver} we obtain 
$$\circled{2}\lesssim \sum_{P\in \mathcal{W}} \fint_P\fint_{5P} \frac{|f(\xi)-f(\zeta)|^p}{ \ell(P)^{s p/2}} \, d\zeta \, d\xi \frac{1}{\ell(P)^{sp/2}} \ell(P)^d \lesssim \sum_{P\in \mathcal{W}} \int_P\sup_{\zeta\in 5P}\frac{|f(\xi)-f(\zeta)|^p}{|\xi-\zeta|^{sp}} \, d\xi, $$
so $\circled{2}\lesssim \norm{f}_{\widetilde A^s_{p,q}(\Omega)}^p$ as claimed. 

To end we  control the last term by similar arguments but with no need to use the chain. In this case, 
$$\circled{3}\leq \sum_Q  \ell(Q)^d \sup_{R: Q\in \SH(R)} \sup_{y \in \SH(R)} \frac{|f_{R}-f(y)|^p}{\ell(R)^{sp}} \leq \sum_{R\in\mathcal{W}}   \sup_{y \in \SH(R)} \frac{|f_{R}-f(y)|^p}{\ell(R)^{sp}} \sum_{Q\in\SH(R)}  \ell(Q)^d,$$
and using \rf{eqMaximalAllOver} and Jensen's inequality, we get
$$\circled{3}\leq \sum_{R\in\mathcal{W}}   \sup_{y \in \SH(R)} \frac{\int_R |f(\xi)-f(y)|^p \,d\xi }{\ell(R)^{sp}} \lesssim \sum_{R\in\mathcal{W}}  \int_R  \sup_{y \in \SH(R)} \frac{ |f(\xi)-f(y)|^p }{|y-\xi|^{sp}} \,d\xi = \norm{f}_{\widetilde A^s_{p,q}(\Omega)}^p.$$
\end{proof}

\begin{theorem}\label{theoBall}
Let $\Omega$ be a uniform domain with an admissible Whitney covering $\mathcal{W}$, let $0<s<1$, $1\leq q\leq p <\infty$. If $f\in L^1_{\rm loc}(\Omega)$ then
$$\left( \sum_{Q\in\mathcal{W}} \int_Q \left(\int_{5Q} \frac{|f(x)-f(y)|^q}{|x-y|^{sq+d}}\, dy \right)^\frac pq dx\right)^\frac1p \approx \norm{f}_{\dot{A}^s_{p,q}(\Omega)}.$$
\end{theorem}
\begin{proof}
The case $1<q\leq p<\infty$ is essentially \cite[Lemma 4.1, Remark 4.2]{PratsSaksman}, with $\Omega$ bounded. The restriction $f\in L^p$ and the boundedness of $\Omega$ are never used along that proof. The same proof can be applied in the case $1 = q \leq p <\infty$  with the usual modifications when the dual exponents are not finite. 
\end{proof}

\renewcommand{\abstractname}{Acknowledgements}
\begin{abstract}
The author was funded by the European Research Council (FP7/2007-2013) under the grant agreement 320501, by 017-SGR-395 (Catalonia) and MTM-2016-77635-P (Spain).

The author wants to thank Kari Astala and Eero Saksman for the discussions that led to the need to show Theorem \ref{theoEquivalent}, Oscar Dom\'inguez for some discussions on the subject and Xavier Tolsa, under whose tuition the author learned most of the techniques developed here. 
\end{abstract}

\bibliography{../../../bibtex/Llibres}
\end{document}